\documentclass[leqno,11pt,a4paper]{amsart}

\usepackage{soul}
\usepackage{jc}
\newcommand{\sys}{\on{sys}}

\newcommand{\RP}{\mathbb{RP}}

\newcommand{\HH}{\mathbb H}

\newcommand{\im}{\mathrm{Im}}

\newcommand{\volHT}{\on{vol}_{\on{HT}}}

\begin{document}
\author{Julian Chaidez}
\address{Department of Mathematics\\University of Southern California\\Los Angeles, CA\\90007\\USA}
\email{julian.chaidez@usc.edu}

\author{Shira Tanny}
\address{Department of Mathematics\\Weizmann Institute of Science\\Rehovot\\76100\\ Israel}
\email{tanny.shira@gmail.com}


\title[The Pu Inequality And The Viterbo Conjecture With Finite Symmetry]{The Pu Inequality And The Viterbo Conjecture With Finite Symmetry}


\vspace*{-30pt}
\maketitle

\vspace*{-25pt}

\begin{abstract} 
We study a construction for producing Finsler metrics on the sphere from convex bodies in $\R^4$ using the quaternions. Among several applications, we prove the weak Viterbo conjecture for convex bodies in $\R^4$ that are invariant under the natural action of the quaternion group $Q$. This is the first instance where Viterbo's systolic inequality is established for convex bodies with a finite group symmetry, addressing a question of Haim-Kislev-Ostrover. Moreover, a smooth $Q$-symmetric convex domain that achieves equality is symplectomorphic to the ball, without any assumption of strict or uniform convexity. Finally, we show that any link of minimal Reeb orbits on the boundary of a uniformly convex $Q$-symmetric convex body is transversely isotopic to a union of Hopf fibers, strengthening a result of Abbondandolo-Edtmair-Kang in the $Q$-symmetric setting. 
\end{abstract}
\section{Introduction} \label{sec:intro}

 In the seminal work \cite{v2000}, Viterbo introduced his eponymous conjectural systolic inequality for convex domains in $\R^{2n}$, which subsequently emerged as one of the fundamental open problems in quantitative symplectic geometry. More specifically, recall that the \emph{symplectic systole}
\[
\sys(X) \in (0,\infty) \qquad\text{of a smooth star-shaped domain }X \subset \R^{2n}
\]
is the minimum symplectic area of a disk $D \subset \R^{2n}$ whose boundary is both contained in $\partial X$ and tangent to the characteristic foliation of $\partial X$. The Viterbo conjecture can be stated as follows.

\begin{conjecture*}[Weak Viterbo] \label{conj:viterbo} Every convex domain $X \subset \R^{2n}$ satisfies the following systolic inequality. 
\begin{equation} \label{eq:viterbo_inequality} \hspace{40pt} \frac{\sys(X)}{\sys(B)} \leq \Big(\frac{\on{vol}(X)}{\on{vol}(B)}\Big)^{1/n} \qquad\text{where $B \subset \R^{2n}$ is the standard $2n$-ball.}\end{equation}
\end{conjecture*}

\noindent We refer to this bound as the \emph{Viterbo inequality}. For convex bodies, the symplectic systole coincides with the Ekeland-Hofer-Zehnder capacity, and a stronger version of the above conjecture states that every symplectic capacity satisfies Conjecture \ref{conj:viterbo} on convex bodies. 

\vspace{3pt}

In \cite{v2000}, Viterbo proved that Conjecture \ref{conj:viterbo} holds up to a multiplicative factor dependent on the dimension. Subsequent work on the conjecture established many special cases \cite{gutt2022examples,cristofaro2023agreement,abbondandolo2018sharp,abbondandolo2023local,abbondandolo2025symplectic,rudolf2022viterbo,karasev2019viterbo}. Connections were found between the Viterbo conjecture, the Mahler conjecture \cite{artstein2014symplectic} and billiard dynamics \cite{artstein2014bounds,balitskiy2020equality,akopyan2016elementary}. Combinatorial methods were established for studying the systoles of polytopes \cite{chaidez2021computing,haim2019symplectic}. Finally, improvements on Viterbo's original result were found using methods from convex geometry \cite{artstein2008m}. However, in spite of the preponderance of evidence, Conjecture \ref{conj:viterbo} was recently refuted by Haim-Kislev-Ostrover \cite{haim2024counterexample}.

\begin{thm*}[Haim-Kislev-Ostrover] There exist convex domains that violate the Viterbo inequality.
\end{thm*}

\noindent More precisely, \cite{haim2024counterexample} illustrated a specific example of a Lagrangian product of two equilateral pentagons whose Ekeland-Hofer-Zehnder capacity violates Conjecture \ref{conj:viterbo}. This remarkable discovery naturally led to several new questions. For instance, the Viterbo inequality is satisfied by convex domains that are sufficiently symmetric (e.g. invariant under the standard $S^1$-action on $\R^{2n} = \C^n$ \cite{gutt2022examples}). This motivates the following question asked by Haim-Kislev-Ostrover.

\begin{question*} \label{qu:vague_symmetry} How much symmetry guarantees that a convex body satisfies the Viterbo inequality?
\end{question*}

In this paper, we give the first example of a finite group, namely the quaternion group $Q$, such that $Q$-symmetric convex bodies satisfy the (sharp) Viterbo inequality in dimension four.  Moreover, a smooth $Q$-symmetric convex domain that achieves equality is symplectomorphic to the ball, without any assumption of strict or uniform convexity. We also prove a new state-of-the-art systolic inequality for convex bodies in $\C^2$ that are invariant under multiplication by $i$, and we reduce the Viterbo inequality for such domains to a conjectural systolic inequality for non-reversible Finsler metrics on the projective plane. These results are applications of a construction for producing Finsler metrics on $S^2$ from convex bodies using the quaternions. 

\subsection{From Convex Bodies To Finsler Metrics} We begin by sketching the construction of Finsler metrics from convex bodies using the quaternions, which will be carried out in Section \ref{sec:quaternion_covering}. Related constructions have appeared previously in \cite{harris2008dynamically,hryniewicz2013global}. Recall that the \emph{quaternion algebra} $\mathbb{H}$ is the algebra over $\mathbb{R}$ generated by three generators
\[
\text{$I$, $J$ and $K$} \qquad\text{with relations}\qquad I^2 = J^2 = K^2 = IJK = -1.
\]
There is a standard inner product, compatible symplectic structure and radial Liouville form induced by the complex structure $I$. These structures all agree with the corresponding structures on $\mathbb{R}^4$ under the standard identification $\mathbb{H} = \mathbb{R}^4$. Finally, recall that the \emph{quaternion group}
\[
Q \subset \mathbb{H}
\]
is the group of order eight generated by $I$ and $J$, which consists of $1,I,J,K$ and their negatives.

\vspace{3pt}

The main tool in this paper is the construction of a Finsler metric on the 2-sphere
\[
F \qquad\text{associated to any centrally symmetric convex body $X \subset \HH$}.
\]
In turn, this construction uses a certain {quaternion covering map} from the group of units of the quaternions to the complement of the zero section in the cotangent bundle of the 2-sphere.
\[\pi:\HH^\times \to T^*S^2 - S^2\]
This map has a number of key properties. First, it is a strict exact symplectic map that respects the natural Liouville forms on both sides. Second, the map commutes with the $Q$-action in the sense that multiplication by $I$ corresponds to the lift of the antipodal map $S^2 \to S^2$ to $T^*S^2$, and multiplication by $J$ corresponds to the map that negates vectors in the cotangent fibers. Finally, the map respects a certain kind of fiberwise convexity, defined as follows. 

\begin{definition*}[Convex Along Lines] A domain $X \subset \mathbb{H}$ is \emph{convex along $J$-complex lines} if $\partial X$ is transverse to the radial vector field, and $L \cap X$ is convex for each $J$-complex line $L$ (i.e. each $J$-invariant linear subspace of dimension two).
\end{definition*}

\noindent Note that $J$-complex lines are Lagrangian with respect to the symplectic form $\Omega$ determined by $I$. These lines play the role of the cotangent fibers upstairs in the covering map, and the map sends a given $J$-complex line to a cotangent fiber. Due to these properties of the quaternion covering, the boundary $\partial X$ of a centrally symmetric domain $X$ that is convex along $J$-complex lines is sent by $\pi$ to the unit cosphere bundle
\[
S^*_FS^2 \qquad\text{of a Finsler metric $F$ on $S^2$.}
\]
Moreover, if the original domain is $I$-invariant, then $F$ descends to the projective plane $\mathbb{RP}^2$ and if the original domain is $J$-invariant, then $F$ is reversible. This yields the desired correspondence.

\begin{figure}[h!]
    \centering
    \includegraphics[width=.9\linewidth]{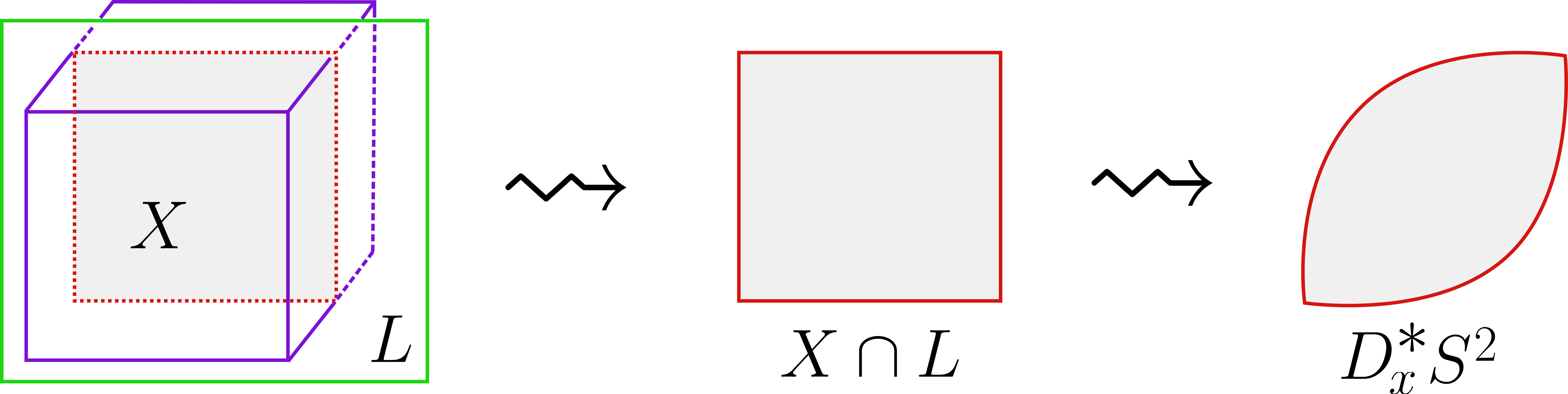}
    \caption{A depiction of the correspondence between bodies and metrics. }
    \label{fig:metric_construction}
\end{figure}

\begin{remark*}[Depiction] A cartoon of the construction is given in Figure \ref{fig:metric_construction}. There, a convex body $X \subset \HH$ is depicted as a 3-dimensional cube. The unit codisk $D^*_xS^2$ of the corresponding metric $F$ at a point $x$ in $S^2$ is given by taking the $J$-complex line $L = \pi^{-1}(T^*_xS^2)$ given by the fiber of $\pi$ over $x$, and taking the image of the slice $X \cap L$ under $\pi$. The map on fibers is (up to linear isomorphism) equivalent to the map $z \mapsto z^2$. \end{remark*}

\begin{remark*}[Generalizations] \label{rmk:generalization_to_fiberwise} Many of our theorems directly generalize to smooth domains that are centrally symmetric and convex along $J$-complex lines. We will state and prove these generalizations in Sections \ref{sec:quaternion_covering}-\ref{sec:applications}.
\end{remark*}

\begin{remark*}[Squared Metrics] A key observation for several of our results is that the dual Finsler metrics associated to convex bodies are \emph{squared} in the following sense. We say that a planar convex body $X$ is \emph{squared} if it is affine equivalent to a convex body of the form $\sigma(U)$ where $U \subset \C$ is a centrally symmetric convex body in the complex plane and $\sigma:\C \to \C$ is the squaring map $\sigma(z) = z^2$. An (asymmetric) norm on a 2-dimensional vector space is squared if its unit ball is squared, and a Finsler metric is squared if the fiberwise norms are squared. 

\vspace{3pt}

Squared convex bodies and norms have several special properties that we will exploit in our applications. For example, a smooth squared norm has a uniformly convex unit ball, and this is needed to show that the fiberwise norms acquired from our Finsler metric construction satisfy the strong convexity property typically required for Finsler metrics (see Definition \ref{def:finsler_metric}). We will also exploit the fact that triangles are \emph{not} squared domains to improve certain systolic estimates, specifically in Theorem \ref{thm:intro_I_invariant_systolic} below. It is interesting to ask if squared Finsler metrics have further special properties that can be exploited in this setting. \end{remark*}

\subsection{Applications To Symplectic Systolic Inequalities Via Pu Inequalities} Our construction of Finsler metrics from centrally symmetric convex bodies permits us to apply metric systolic inequalities on the projective plane to prove new systolic inequalities for symmetric convex bodies in dimension four. Let us now discuss these applications. Detailed proofs will be given in Section \ref{sec:applications}.

\vspace{3pt}

We first consider domains $X$ that are invariant under the action of the quaternion group $Q$. In this case, the body corresponds to a reversible Finsler metric on $\mathbb{RP}^2$. Thus by applying the Pu inequality for reversible Finsler metrics proven by Ivanov \cite{ivanov2011filling}, we acquire the following theorem. 

\begin{thm*}[$Q$-Symmetric Viterbo Conjecture] \label{thm:intro_Q_Viterbo}
    Let $X$ be a convex domain in $\mathbb{H}$ that is invariant under the action of the quaternion group $Q$ by right multiplication. Then $X$ satisfies the Viterbo inequality. 
\end{thm*}

A stronger version of the Viterbo conjecture stated that any convex body achieving equality in the Viterbo inequality has interior symplectomorphic to the open ball. The counter-example of Haim-Kislev-Ostrover \cite{haim2024counterexample} can be used to easily construct a counter example to this conjecture. On the contrary, our methods show that this conjecture holds in the smooth $Q$-symmetric setting.

\begin{thm*}[$Q$-Symmetric Optimizers] \label{thm:intro_Q_optimizers} A $Q$-invariant smooth convex domain $X$ in $\mathbb{H}$ achieves equality in the Viterbo inequality if and only if $X$ is symplectomorphic to the ball $B$.\end{thm*}

\noindent It was previously known that any smooth and strictly convex domain in $\R^{2n}$ that locally maximizes the systolic ratio must have Zoll Reeb flow on its boundary\cite{AlvarezPaivaBalacheff2014}, and thus must be a ball in dimension four \cite{AbbondandoloBramhamHryniewiczSalomao2018}. The novelty of Theorem \ref{thm:intro_Q_optimizers} is that it drops any assumption of strict or uniform convexity, which is a property required by previous results of this type. We expect that our methods may also be applicable to the full (non-smooth) case. We will discuss this later in the introduction.

\vspace{3pt}

We next consider bodies $X$ that are invariant only under multiplication by $I$. In this case, the body corresponds to a non-reversible Finsler metric on $\mathbb{RP}^2$. It is conjectured (cf. \cite{alonsooptimal,alonso2024survey}) that the Ivanov-Pu inequality extends to non-reversible metrics, yielding a conditional result.

\begin{thm*}[Conditional $I$-Symmetric Viterbo]\label{thm:intro_conditional} If the Ivanov-Pu inequality extends to non-reversible Finsler metrics on $\mathbb{RP}^2$, then any $I$-symmetric convex body in $\HH$ satisfies the Viterbo inequality. Moreover, equality occurs in the smooth case if and only if the body is symplectomorphic to a ball.
\end{thm*}

\noindent Regardless, we can use a symmetrization argument and a stable version of the Rogers-Shephard inequality in convex geometry to prove the following state-of-the-art bound.

\begin{thm*}[$I$-Symmetric Systolic Bound] \label{thm:intro_I_invariant_systolic} Let $X$ be a convex domain in $\HH$ that is invariant under right multiplication by $I$. Then $X$ satisfies the following systolic inequality.
\[
\frac{\sys(X)}{\sys(B)} \le C \Big(\frac{\on{vol}(X)}{\on{vol}(B)}\Big)^{1/2} \qquad\text{where $C$ is a constant less than}\qquad \Big(\frac{3}{2}\Big)^{1/2} \simeq 1.2247
\]
\end{thm*}

\begin{remark*} The best previous bound for $I$-invariant convex bodies, given by Gutt-Hutchings-Ramos in \cite{gutt2022examples} following an argument of Ostrover, gives the following constant.
\[
4/\pi \simeq 1.2732
\]
We will review the proof in Section \ref{sec:applications}, since part of that result will be useful for Theorem \ref{thm:intro_I_invariant_systolic} . \end{remark*}
\subsection{Applications To Hopf Orbits And Orbit Links} A natural family of questions arising in the dynamics of flows on 3-manifolds deals with the types of knots and links formed by periodic orbits. This has been studied extensively in the setting of Reeb flows, and there are interesting connections to other properties such as entropy \cite{AlvesPirnapasov2022ForceEntropy, Meiwes2025LinkComplements, AlvesDahindenMeiwesPirnapasov2026Stability}. In general, there are no constraints on knot types of a Reeb flow even on the standard contact 3-sphere \cite{EtnyreGhrist2000KnottedOrbits}. However, many constraints are known for convex and dynamically convex Reeb flows \cite{HryniewiczSalomaoSiefring2025KnotTypes, HainzHamenstaedt2011TopologicalProperties, CiobaWendl2020Unknotted, HoferWysockiZehnder1996Unknotted, HoferWysockiZehnder1998Dynamics, HoferWysockiZehnder2003Foliations, Hryniewicz2012FastPlanes, HryniewiczSalomao2011DiskLikeSections, Hryniewicz2014Systems, abbondandolo2024closed, HryniewiczHutchingsRamos2026HopfOrbits, Edtmair2024DiskLikeSurfaces}.

\vspace{3pt}

Our construction can be used to give highly simplified proofs of the existence of Hopf orbits and links in the setting of $Q$-symmetric convex bodies. Recall the following definition.

\begin{definition*}[Hopf] A transverse link $L$ in the standard contact $S^3$ is \emph{Hopf} if it is transversely isotopic to a union of fibers of the Hopf fibration. 
\end{definition*}

\noindent Note that any two Hopf links with the same number of components are transversely isotopic, and a Hopf knot is simply the standard unknot of self-linking number $-1$ by work of Eliashberg \cite{eliashberg1993legendrian}. Hopf orbits play a prominent role in Reeb dynamics and quantitative symplectic geometry in dimension three. For instance, we have the following beautiful result of Hryniewicz.

\begin{thm*} \label{thm:hryniewicz} \cite{hryniewicz2013global,Hryniewicz2012FastPlanes} A closed Reeb orbit $\Gamma$ of a dynamically convex contact form on the standard contact $3$-sphere $S^3$ bounds a disk-like Birkhoff section if and only if $\Gamma$ is Hopf.
\end{thm*}

\noindent It was recently shown by Abbondandolo-Edtmair-Kang \cite{abbondandolo2024closed}  that any minimizer on the boundary of a smooth, uniformly convex domain in dimension four is Hopf. 

\vspace{3pt}

Our next result is a strengthening of the main result of \cite{abbondandolo2024closed} in the $Q$-symmetric setting.

\begin{thm*}[$Q$-Symmetric Hopf Link] \label{thm:intro_hopf_link}
    Let $X$ be a $Q$-symmetric smooth, uniformly convex domain in $\HH$ and let $L \subset \partial X$ be any link consisting of minimal period closed Reeb orbits. Then $L$ is Hopf.
\end{thm*}

Our proof is completely independent of the proof in \cite{abbondandolo2024closed} and much simpler, avoiding any use of Floer homology or Morse theory with the Clarke dual functional. The main idea is to prove the desired result on the Finsler metric side, where the link type of a collection of lifted geodesics in the unit sphere bundle of $\mathbb{RP}^2$ can be described using the collection of geodesics as a set of immersed curves in $\mathbb{RP}^2$ itself. This point of view is used, for instance, in the work of Dehornoy-Rechtman \cite{dehornoy2025left}. One may show that any minimizing geodesic in $\mathbb{RP}^2$ is embedded and that any two such geodesics intersect once. One may then show that the cotangent lift of a set of geodesics satisfying these properties must necessarily be isotopic to a Hopf link.

\begin{remark} We will also provide a simplified proof that every (individual) minimal periodic Reeb orbit is Hopf in a uniformly convex centrally symmetric domain $X$. This proof uses Clarke duality and the correspondence with geodesics. It is also independent of the proof in \cite{abbondandolo2024closed}.
\end{remark}

\noindent  Theorems \ref{thm:hryniewicz} and \ref{thm:intro_hopf_link} imply the following corollary, which is the main result of \cite{abbondandolo2024closed} in the $Q$-symmetric case.

\begin{cor*}[$Q$-Symmetric Birkhoff Section] \label{thm:intro_section} Let $X \subset \HH$ be a $Q$-symmetric, smooth, uniformly convex domain. Then every minimal period Reeb orbit $K \subset \partial X$ bounds a disk-like Birkhoff section $D$ such that every other minimal Reeb orbit corresponds to a fixed point of the associated return map $\Phi:D \to D$.
\end{cor*}

\noindent We conclude with a discussion of questions and conjectures arising from this work. 
\subsection{Questions On Low Regularity Optimizers}\label{subsec:questions_low_regularity}   A first interesting problem is to extend Theorem \ref{thm:intro_Q_optimizers} to arbitrary $Q$-invariant convex bodies that are not necessarily smooth. Evidently a non-smooth convex body cannot be smoothly symplectomorphic to the closed ball. Thus we propose the following modification of Theorem \ref{thm:intro_Q_optimizers}. 

\begin{conjecture*}[General Optimizers] \label{conj:general_optimizers} A $Q$-invariant, possibly non-smooth, convex domain $X$ in $\mathbb{H}$ achieves equality in the Viterbo inequality if and only if the interior of $X$ is symplectomorphic to a ball.
\end{conjecture*}

\noindent Our covering construction provides a Finsler geometry perspective on Conjecture \ref{conj:general_optimizers}. Specifically, we say that a Finsler metric $F$ is \emph{systolically optimal} if the systole and volume agree with those of a Zoll metric on $\mathbb{RP}^2$. The following conjecture on systolically optimal metrics implies Conjecture \ref{conj:general_optimizers}. 


\begin{conjecture*}[Metric Approximation] Let $F$ be a reversible Lipschitz Finsler metric on $\mathbb{RP}^2$ that is systolically optimal. Then $F$ is the $C^0$-limit of a sequence of smooth Zoll reversible Finsler metrics on $\mathbb{RP}^2$.
\end{conjecture*}

\noindent It seems potentially relevant that the space of smooth reversible Zoll metrics on the projective plane is contractible by work of Sabourau \cite{sabourau2016strong}. 

\vspace{3pt}

Next, there are also generalized notions of Zollness in the low regularity setting of general convex bodies, introduced recently by Matijevi\'{c} \cite{matijevic2024positive,matijevic2025systolic} and Haim-Kislev \cite{haim2025dynamical}, which are formulated using the theory of generalized Reeb trajectories (cf. \cite{artstein2014bounds,artstein2014symplectic,chaidez2021computing,haim2019symplectic}). Convex bodies whose interiors are symplectomorphic to balls are generalized Zoll by \cite[Cor 1.8]{matijevic2025systolic}. Therefore the following conjecture is a weakening of Conjecture \ref{conj:general_optimizers}. 

\begin{conjecture*}[Zoll Optimizers] \label{qu:zoll_optimizers} Let $X$ be a $Q$-invariant convex domain in $\mathbb{H}$ that achieves equality in the Viterbo inequality. Then $X$ is generalized Zoll in the sense of Matijevi\'{c} \cite{matijevic2025systolic}.
\end{conjecture*}

\noindent Generalized Zollness is a property of the homotopy type of the space of generalized systoles on the boundary of a convex body, and through the constructions of this paper, one can instead study the space of shortest oriented non-contractible $C^1$-loops
\[
\on{Sys}(F) \qquad\text{for a reversible Finsler metric $F$ on $\mathbb{RP}^2$.}
\]
A convex body in $\R^4$ whose space of systoles is homotopy equivalent to $S^2$ is generalized Zoll \cite{matijevic2025systolic}. Therefore the following conjecture seems natural and should have applications to Conjecture \ref{qu:zoll_optimizers}.

\begin{conjecture*}[Systolic Homotopy Type] \label{qu:zoll_homotopy_type} Let $F$ be a reversible Lipschitz Finsler metric on $\mathbb{RP}^2$ that is systolically optimal. Then the space $\on{Sys}(F)$ is homotopy equivalent to $S^2$. 
\end{conjecture*}

\begin{remark*} There are many further weakenings of Conjectures \ref{qu:zoll_optimizers} and \ref{conj:general_optimizers} that one could also consider. In particular, we expect that the methods of this paper could be extended to show that on a $Q$-invariant convex body attaining equality in the Viterbo inequality, every point must have a closed generalized Reeb trajectory of minimal period through it. 
\end{remark*}

\subsection{Questions On Centrally Symmetric Bodies} We finally consider convex bodies that are only centrally symmetric. This is, in some sense, the largest class of convex bodies that could satisfy the Viterbo inequality, and the Viterbo conjecture for centrally symmetric bodies is strictly stronger than the Mahler conjecture. This case is thus singled out by Haim-Kislev-Ostrover.

\begin{question*} \cite[Quest 1]{haim2024counterexample} \label{qu:symmetric_viterbo} Do centrally symmetric convex bodies satisfy the Viterbo inequality?
\end{question*}

A centrally symmetric domain corresponds to a (possibly) non-reversible Finsler metric on the sphere under our construction. A famous result of Croke \cite{croke1988area} states that the length of the shortest closed geodesic for any Riemannian metric on the sphere is bounded above by a constant multiple of the square root of the area. This can be extended to Finsler metrics, yielding the following result.

\begin{thm*}[Viterbo For Fiberwise Convex] \label{thm:intro_sys_ineq_fiberwise_convex} Let $X$ be a smooth, centrally symmetric domain in $\HH$ that is convex along $J$-complex lines. Then there is a constant $C > 0$ such that
\[
\frac{\sys(X)}{\sys(B)} \le C \Big( \frac{\on{vol}(X)}{\on{vol}(B)}\Big)^{1/2}
\]
\end{thm*} 

\noindent On the other hand, the analogue of the Viterbo conjecture is false in the setting of Finsler (or even Riemannian) metrics. Indeed, Croke \cite{croke1988area} himself pointed out an example of a singular metric on the sphere due to Calabi whose systolic ratio (or more precisely, diastolic ratio) exceeds that of the standard round sphere. It is natural to ask if such a metric counter-example could yield a counter-example to Question \ref{qu:symmetric_viterbo}.

 \begin{question*} \label{qu:viterbo_violator} Is there a Riemannian (or Finsler) metric on the sphere, coming from a centrally symmetric convex body in $\HH$, whose systolic ratio exceeds that of the round metric?
\end{question*}

There are very few explicit examples of Finsler metrics in the literature that have a longer shortest geodesic than the round metric of the same volume. Thus it is an interesting, but seemingly subtle, problem to construct an example as sought in Question \ref{qu:viterbo_violator}.

\begin{remark*}[Searching For Examples] We performed some investigation of Question \ref{qu:viterbo_violator} during the preparation of this paper. In particular, we computed the domain $X \subset \HH$ corresponding to the Calabi-Croke metric under our correspondence, and determined that its systolic ratio (i.e. its volume normalized systole) is actually smaller than that of the ball. Moreover, we found that this body is only convex along $J$-complex lines, but not convex. Thus it seems unlikely that an example for Question \ref{qu:viterbo_violator} will exist in a neighborhood near the Calabi-Croke metric.
\end{remark*}

\subsection{The Symmetric Capacity Conjecture} Finally, a longstanding stronger variant of the Viterbo conjecture stated that all normalized symplectic capacities agree on convex bodies in $\mathbb{R}^{2n}$. The counter-example of Haim-Kislev-Ostrover \cite{haim2024counterexample} also refuted this stronger conjecture. However, in light of Theorem \ref{thm:intro_Q_Viterbo}, it is natural to pose this conjecture again with a symmetry assumption.

\begin{question*}[Capacity Conjecture] Does the Ekeland-Hofer-Zehnder capacity coincide with the Gromov width for $Q$-symmetric convex bodies in $\HH$?
\end{question*}

\subsection*{Outline} In Sections \ref{sec:finsler_metrics} and \ref{sec:systolic_geometry}, we discuss various background results from Finsler geometry and systolic geometry. In Section \ref{sec:quaternion_covering}, we present the main covering construction described informally above. In Section \ref{sec:applications}, we present the proofs of the various applications.

\subsection*{AI Disclosure.} ChatGPT assisted in formalizing the construction in Section \ref{sec:quaternion_covering}. ChatGPT also proposed Lemma \ref{lem:chord_bound}. The main ideas of this paper and all other mathematical content are due solely to the authors. This article contains no AI generated text, but AI tools were used to proofread the article.

\subsection*{Acknowledgments} 
Some of this work was conducted during the joint stay of JC and ST at the Institute for Advanced Study for the Special Year in Conformally Symplectic Dynamics. The authors are grateful for their hospitality. JC and ST were partially supported by US-Israel BSF Award 2024157. JC was partially supported by NSF Award DMS-2446019. 
ST was partially supported by the Center for New Scientists at the Weizmann
Institute of Science and ISF Beresheet grant 4097/25.

\section{Finsler Geometry} \label{sec:finsler_metrics}

In this section, we review the basic concepts and results from Finsler geometry that will be needed in this paper. We refer the reader to the monograph of Bao-Chern-Shen \cite{bao2012introduction} for a more detailed exposition.

\subsection{Asymmetric Norms} We start by reviewing key concepts regarding asymmetric norms on vector spaces. In particular, we review the notion of strong convexity for (possibly non-smooth) asymmetric norms.

\begin{definition}[Norms] An \emph{asymmetric norm} $F$ on a vector space $V$ is a positively 1-homogeneous function that is positive away from the origin and that is convex in the sense that
\[
F(u + v) \le F(u) + F(v) \qquad\text{for all $u,v \in V$.}
\]
A \emph{(symmetric) norm} is an asymmetric norm that is fully 1-homogeneous, or equivalently
\[
F(v) = F(-v) \qquad\text{for all $v$ in $V$.}
\]\end{definition}

\begin{remark}[Regularity] An asymmetric norm will be called smooth, differentiable or Lipschitz if it has the respective regularity on the complement of the origin. We will primarily consider smooth norms in this paper.
\end{remark}

We will be interested in asymmetric norms satisfying the following strong convexity property.

\begin{definition}[Strong Convexity] \label{def:strong_convexity} An asymmetric norm $F$ on a vector space $V$ is \emph{strongly convex} if there is a positive constant $\epsilon$ such that
\begin{equation}\label{eq:strong_convexity_def}
    \epsilon F^2(u-v) \le F^2(u) + F^2(v) - \frac{1}{2}F^2(u+v)
\end{equation}
The \emph{strong convexity factor} $\epsilon(F)$ is the maximal $\epsilon$ such that this inequality holds for all $u$ and $v$. \end{definition}

 We will need several basic properties of strong convexity. First, note that we have the following closedness property.

 \begin{remark}[Closedness] \label{rmk:openess} Note that the set of asymmetric norms that are strongly convex with factor greater than or equal to $\epsilon$ is closed in the compact-open topology on norms.
 \end{remark}

\noindent Next, we note that strong convexity of a given norm implies strict convexity of the corresponding closed unit ball.

\begin{lemma}[Strict Convexity] Let $F$ be a strongly convex asymmetric norm on a vector space $V$. Then $F$ (or equivalently, the closed unit ball $D_F$) is strictly convex.
 \end{lemma}

 \begin{proof} Let $u$ and $v$ be distinct vectors in $D_F$ so that $F(u) = F(v) = 1$ and let $s \in (0,1)$. Then
 \[
 F(\frac{u+v}{2})^2 < 2F^2(\frac{u}{2}) + 2F^2(\frac{v}{2}) \le \frac{1}{2}(F(u)^2 + F(v)^2) = 1
 \]
 Thus the midpoint of any line segment between points on the boundary of $D_F$ is in the interior of $D_F$. In other words, $D_F$ is strictly convex.\end{proof}

 \noindent Next, we have the following useful equivalent characterizations in the case of smooth norms.
 
\begin{lemma}[Smooth Criteria] \label{lem:smooth_convex_criterion} For any smooth asymmetric norm $F$ on $\R^n$, the following are equivalent.
\begin{itemize}
    \item[(a)] The asymmetric norm $F$ is strongly convex in the sense of Definition \ref{def:strong_convexity}. 
    \vspace{2pt}
    \item[(b)] The Hessian of $F^k$ is positive definite at every non-zero point $x$ in $\R^n$ for some constant $k > 1$.
    \vspace{2pt}
    \item[(c)] The Hessian of $F^k$ is positive definite at every non-zero point $x$ in $\R^n$ for every constant $k > 1$.
    \vspace{2pt}
    \item[(d)] The closed unit disk $D_F \subset \R^n$ of $F$ is uniformly convex, or equivalently the unit sphere $\partial D_F \subset \R^n$ has positive definite second fundamental form.
\end{itemize}
\end{lemma}

\begin{proof} We break the proof into several steps. The first two steps show that (b-d) are equivalent. The final step will show that (a) is equivalent to the other conditions.

\vspace{3pt}

\noindent {\bf Step 1.} We start by showing that (d) implies (c). Assume that $Y = \partial D_F$ is uniformly convex, or equivalently that the second fundamental form $S$ of $Y$ is positive definite at every point $x$ in $Y$. First, note that since $F$ is 1-homogeneous, it follows that
\begin{equation} \label{eq:smooth_convex_criterion:1} d_xF(x) = F(x) \qquad\text{and}\qquad D^2_xF(x,u) = 0 \end{equation}
for any non-zero $x$ in $\R^n$ and any vector $u$ in $\R^n$. Therefore $1$ is a regular value of $F$ and the second fundamental form at any point in $Y = F^{-1}(1)$
\begin{equation} \label{eq:smooth_convex_criterion:2}
S_x = \frac{1}{|DF_x|} \cdot D^2_x F|_{TY} \qquad\text{is positive definite for each non-zero $x$ in $Y$.}
\end{equation}
Second, note that the Hessian of $F^k$ is given by the following formula.
\begin{equation} \label{eq:smooth_convex_criterion:3}
D^2 F^k = k F^{k-1} D^2 F + k(k - 1) F^{k-2}dF \otimes dF
\end{equation}
Finally, note that $Y$ is transverse to the radial vector-field $Z$ given by $Z(x) = x$, and thus that $T_x\R^n$ splits as a direct sum $\R \oplus T_xY$ where $\R$ is spanned by the radial vector-field $Z$. It  follows from (\ref{eq:smooth_convex_criterion:1}-\ref{eq:smooth_convex_criterion:3}) that under this splitting, the Hessian $D^2_xF^k$ is identified with the direct sum of the rank one form $k(k-1)F^{k-2}$ and a positive scalar multiple of $S_x$. Therefore the Hessian $D^2_xF^k$ is positive definite at each $x$ in $Y$. Since every non-zero point can be scaled to lie on the unit sphere $Y$, it follows from 1-homogeneity that it is positive definite away from the origin.

\vspace{3pt}

\noindent {\bf Step 2.} Next we show that (b) implies (d). Since (c) clearly implies (b), this will imply that (b-d) are equivalent. Here we simply note that by (\ref{eq:smooth_convex_criterion:2}) and (\ref{eq:smooth_convex_criterion:3}), we know that
\[
S_x = \frac{1}{k |DF_x|} \cdot D^2_x F^k|_{TY} \qquad\text{for any $x$ in $Y$.}
\]
Therefore if $D^2_x F^k$ is positive definite, the second fundamental form $S_x$ is also positive definite.

\vspace{3pt}

\noindent {\bf Step 3.} Finally, we show that (a) is equivalent to (b-d). Specifically, we show that $F$ is strongly convex if and only if the Hessian of $F^2$ is positive definite. 

\vspace{3pt}

We first derive a useful integral expression for the inequality in Definition \ref{def:strong_convexity}. Fix vectors $u,v$, and consider the function $h:[0,1] \to \R$ given by 
$$h(s):= F(u(s))^2\quad \text{where} \quad u(s) = (1 - s)u + sv.$$
This function is smooth for all $s$ such that $u(s)\neq 0$. We have 
\[
F^2(u) + F^2(v) - \frac{1}{2}F^2(u+v) = h(0) + h(1) - 2h(1/2)
\]
On the other hand, we can use integration by parts to see that
\[
 h(0) + h(1) - 2h(1/2) 
= \int_{0}^{1/2} s h''(s) ds + \int_{1/2}^1 (1-s)h''(s) ds
\]
Here, if necessary, we remove the point $s$ such that $u(s)=0$ from the domain of integration.
The second derivative of $h(s)$ can be written using the Hessian of $F^2$ as follows.
\[
h''(s) = D^2(F^2_{u(s)})(v-u,v-u)
\]
Therefore, we have the following integral expression where $\chi(s) = \on{min}(s,1-s)$.
\begin{equation} \label{eq:smooth_convex_criterion:4}
F^2(u) + F^2(v) - \frac{1}{2}F^2(u+v) = \int_0^1 \chi(s) \cdot D^2(F^2_{u(s)})(v-u,v-u) \cdot ds
\end{equation} 

Now we prove the claimed equivalence. First, if $D^2 (F^2)$ is positive definite at every non-zero point, then there is an $\epsilon > 0$ such that $4\epsilon F^2(w) \le D^2_x(F^2)(w,w)$ for every non-zero $x$ and every $w$. We want to prove that the strong convexity inequality (\ref{eq:strong_convexity_def}) holds.  It follows from (\ref{eq:smooth_convex_criterion:4}) that
\[
F^2(u) + F^2(v) - \frac{1}{2}F^2(u+v)  \ge 4\epsilon \cdot F^2(u-v) \cdot \int_0^1 \chi(s) ds = 4\frac{1}{4}\epsilon \cdot F^2(u-v)
\]
Therefore $F$ is strongly convex with factor $\epsilon$. Conversely, fix $\epsilon$ and suppose that there exists a non-zero $x$ in $V$ and a vector $w$ such that
\[
D^2F^2_x(w,w) < 4\epsilon F^2(w).
\]
Consider the paths $u(s) = x + sw$ and $v(s) = x - sw$. Then it is simple to see using Taylor series that for all $s$ such that $u(s),v(s)\neq 0$,
\[
F^2(u(s)) + F^2(v(s)) - \frac{1}{2}F(u(s) + v(s))^2 =s^2 \cdot D^2F^2_x(w,w) + O(s^3) < s^24\epsilon F^2(w) + O(s^3)
\]
On the other hand, $F(u(s) - v(s))^2$ is given by
\[
F(u(s) - v(s))^2 = F(2sw)^2 = 4s^2 \cdot F(w)^2
\]
It follows that for every $\epsilon$ and sufficiently small $s$, we have
\[
F^2(u(s)) + F^2(v(s)) - \frac{1}{2}F^2(u(s) + v(s))< s^24\epsilon F^2(w) = \epsilon F^2(u(s)-v(s))
\]
Therefore $F$ is not strongly convex with factor $\epsilon$. This concludes the step and the proof of the lemma. \end{proof}

\noindent The above characterizations of strong convexity can be converted into formulas and bounds for the corresponding strong convexity factor.

\begin{lemma}[Smooth Factor] \label{lem:smooth_factor} Let $F$ be a smooth strongly convex asymmetric norm on $\R^n$. Then the convexity factor is given by
\[
\epsilon(F) = \frac{1}{4}\on{min}\big(D_x^2F^2(v,v) \; : \; x,v \in \partial D_F\big)
\]
Moreover, let $\kappa$ denote the minimum principal curvature of $\partial D_F$ and suppose that $r,R$ are constants such that $r F_{\on{std}} \le F \le R F_{\on{std}}$ where $F_{\on{std}} = |\cdot|$ is the Euclidean norm. Then
\[
\epsilon(F) \ge \frac{\kappa r^2}{8R^2(r+ \kappa)}
\]\end{lemma}

\begin{proof} The first formula follows immediately from the argument in Step 3 of Lemma \ref{lem:smooth_convex_criterion}. Indeed, that argument shows that
\begin{equation} \label{eq:smooth_factor:1}
\epsilon(F) = \frac{1}{4}\on{max}\big(\epsilon \; : \; D_x^2F^2(v) \ge \epsilon F^2(v) \text{ for non-zero }x,v \in V\big)
\end{equation}
Since $F$ is 1-homogeneous, $F^2$ is 2-homogeneous and $D^2F^2$ is 0-homogeneous (as a tensor). Thus we can replace (\ref{eq:smooth_factor:1}) with a minimum over $x,v$ in $\partial D_F$. This demonstrates the first formula. 

\vspace{3pt}
For the second formula, fix $x$ and $v$ in $Y = \partial D_F \subset \R^n$. We may decompose $v$ as a sum $v = tx + w$ where $w \in T_xY$. Then the formulas (\ref{eq:smooth_convex_criterion:1}-\ref{eq:smooth_convex_criterion:3}) imply that
\begin{equation} \label{eq:smooth_factor:2}
D^2F^2_x(v,v) = 2t^2 + 2|DF_x| \cdot S_x(w,w) \ge 2(t^2 + \kappa |DF_x| |w|^2)
\end{equation}
Next, we apply the following weighted Cauchy-Schwarz inequality.
\[
|u_1 + u_2|^2 \le (a_1 + a_2)(\frac{|u_1|^2}{a_1} + \frac{|u_2|^2}{a_2})
\]
By taking $u_1 = tx$, $u_2 = w$, $a_1 = |x|^2$ and $a_2 = (\kappa|DF_x|)^{-1}$, we acquire the following lower bound.
\[|v|^2 = |tx + w|^2 \le \big(|x|^2 + \frac{1}{\kappa |DF_x|}\big) \cdot \big(t^2 + \kappa|DF_x||w|^2\big) \]
Rearranging this bound and combining it with (\ref{eq:smooth_factor:2}) yields the following estimate.
\begin{equation} \label{eq:smooth_factor:3}
D^2F^2_x(v,v) \ge \frac{\kappa |v|^2}{2(|DF_x|^{-1} + \kappa |x|^2)} 
\end{equation}
Finally, note that we have the following estimates on the terms appearing in (\ref{eq:smooth_factor:3}). 
\[
|v| \ge \frac{1}{R}F(v) = \frac{1}{R} \qquad |x| \le \frac{1}{r}F(x) = \frac{1}{r}\qquad \frac{1}{|DF_x|} = \langle x,\frac{DF_x}{|DF_x|}\rangle \le |x| \le \frac{1}{r}
\]
Combining these bounds with (\ref{eq:smooth_factor:3}) yields the desired second formula.\end{proof}

Finally, we have the following duality property for strongly convex asymmetric norms. This is a standard result and a proof can be found in \cite[Prop 14.8.1]{bao2012introduction}.

\begin{lemma}[Duality] \label{lem:stron_costrong_duality} The dual norm $F^*$ of a smooth strongly convex asymmetric norm $F$ is smooth and strongly convex.
\end{lemma}

\subsection{Squared Norms And Metrics} We next introduce a special class of norms that arise naturally in our constructions. Precisely, consider the \emph{squaring map} on the complex plane
\begin{equation}
\sigma:\C \to \C \qquad\text{given by}\qquad \sigma(z) = z^2
\end{equation}

\begin{definition}[Squared Domain] \label{def:squared_domain} A domain $L$ in a 2-dimensional affine space $E$ is \emph{squared} if there is a centrally symmetric convex domain $K \subset \C$ and an affine isomorphism $A:E \to \C$ with
\[
A(L) = \sigma(K)
\]
\end{definition}

\begin{definition}[Squared Norm] An asymmetric norm $F$ on a 2-dimensional vector space $V$ is \emph{squared} if the unit ball $D_F$ is a squared convex body.
\end{definition}
The following key lemma states that any squared norm is automatically strongly convex. 

\begin{lemma} \label{lem:square_stays_convex} Let $F$ be a squared asymmetric norm on a 2-dimensional real vector space $V$. Then $F$ is strongly convex.
\end{lemma}

\begin{proof} Consider the closed unit ball $D_F \subset V$. By assumption, there is an affine isomorphism
\[
A:V \to \C \qquad\text{such that}\qquad A(D_F) = \sigma(K)
\]
for some centrally symmetric convex body $K$, which necessarily contains the origin in its interior. Let $\rho$ and $\tau$ denote the radial functions parametrizing the boundary of $K$ and $\sigma(K)$ respectively.
\[\rho:S^1 \rightarrow \R_+ \qquad\text{and}\qquad \tau:S^1 \rightarrow \R_+\]
Recall that these functions are uniquely defined by the property that the ray of angle $\theta$ through the origin intersects $\partial K$ at radius $\rho(\theta)$ and $\partial(\sigma(K))$ at $\tau(\theta)$. Since $K$ is centrally symmetric and $\sigma$ is the squaring map, the radial functions satisfy
\[\rho(\theta+\pi) = \rho(\theta) \qquad\text{and}\qquad \tau(\theta) = \rho(\theta/2)^2\]
We now break the proof into the case where $K$ is smooth and the general case. 

\vspace{3pt}

\noindent {\bf Smooth Case.} First assume that $K$ is smooth. In this case, the convexity of $K$ is equivalent to the non-negativity of the (scalar) principal curvature of $\partial K$. Under the parametrization determined by $\rho$, this is
\[
\kappa_K = \frac{\rho^2 +2\dot \rho^2 - \rho \ddot{\rho}}{(\rho^2 + \dot{\rho}(\theta)^2)^{3/2}}\geq 0.
\]
The strong convexity of $\sigma(K)$ is equivalent to the positivity of curvature of the curve defined by $\tau$, which we verify as follows.
\[\tau^2 +2\dot \tau^2  - \tau \ddot{\tau}  = \rho^4 +2 (\rho \dot \rho)^2 - \rho^2\left(\frac{1}{2} \dot \rho^2 +\frac{1}{2}\rho \ddot{\rho}\right)\]
\[= \frac{\rho^2}{2}\left(2\rho^2 + 3 \dot\rho^2 -\rho \ddot{\rho}\right) = \frac{\rho^2}{2}\left((\rho^2 + 2 \dot\rho^2 -\rho \ddot{\rho}) + (\rho^2 +  \dot\rho^2)\right)\geq \frac{\rho^2}{2}\left(0 + (\rho^2 +  \dot\rho^2)\right) \geq \frac{\rho^4}{2}\]
This implies that the principal curvature $\kappa_L$ of the squared body $L = A(D_F) = \sigma(K)$ is bounded below as follows.
\begin{equation} \label{eq:square_convex_1}
\kappa_L \ge \frac{\tau^2}{2(\tau^2 + \dot{\tau}^2)^{3/2}} > 0
\end{equation}
\noindent {\bf General Case.} Now consider the general case. We assume without loss of generality that $V = \R^2$. 

\vspace{3pt}

Choose a sequence of centrally symmetric, smooth convex bodies $K_i \subset \C$ converging to $K$ in the Hausdorff sense. Let $L_i = \sigma(K_i)$ be the corresponding squared bodies with radial functions $\tau_i$ and principal curvatures $\kappa_i$. Note that
\[
\tau_i \to \tau \qquad\text{uniformly on $S^1$.}
\]
Moreover, since $K$ is convex, the radial functions $\rho$ and $\tau$ are Lipschitz and there is a constant $C$ such that $\tau_i$ is $C$-Lipschitz for all $i$. It follows from (\ref{eq:square_convex_1}) that the principal curvatures $\kappa_i$ are uniformly bounded below by a positive constant. 

\vspace{3pt}

Next, let $F_i$ be the smooth asymmetric norm on $V = \R^2$ with unit ball $D_i = A^{-1}(L_i)$. Then $F_i$ converges to $F$ in the compact-open topology. Moreover, the principal curvatures of the unit spheres $\partial D_i$ are uniformly bounded below by a constant $c > 0$. Moreover, there are constants $r$ and $R$ such that
\[
rF_{\on{std}} \le F_i \le RF_{\on{std}} \qquad\text{for all $i$.}
\]
where $F_{\on{std}}$ is the standard Euclidean norm. Therefore, Lemma \ref{lem:smooth_factor} implies that the strong convexity factors of $F_i$ satisfy
\[
\epsilon(F_i) \ge \frac{cr^2}{8R^2(r + c)}
\]
By Remark \ref{rmk:openess}, $\epsilon(F)$ also satisfies this lower bound and thus $F$ is strongly convex. \end{proof}

\subsection{Finsler Metrics} \label{subsec:finsler_metrics} We next recall the notion of a Finsler metric, and the associated notion of Holmes-Thompson volume. 

\begin{definition}[Finsler Metric] \label{def:finsler_metric} A \emph{Finsler metric} $F$ on a manifold $M$ is a smooth function on the tangent bundle $TM\setminus 0$ that restricts to a strongly convex asymmetric norm on each tangent space. 
\end{definition}

\begin{remark}[Regularity] A Finsler metric will be called continuous or differentiable if it has the corresponding regularity as a function on $TM\setminus 0$. A Finsler metric is Lipschitz if the  corresponding function on $TM\setminus 0$ is Lipschitz and the fibers are convex, rather than strictly convex.
\end{remark}

There are several inequivalent notions of volume in the Finsler setting. In this paper, we focus on the following variant that arises naturally from the symplectic point of view.

\begin{definition}[Holmes-Thompson Volume]
    The \emph{Holmes-Thompson volume} of a manifold $M$ with respect to a Finsler metric $F$ is the symplectic volume of the unit codisk bundle of $F$.
\[
\volHT (M,F) := \frac{1}{n!} \cdot \int_{D^*_F M}\Omega_{\on{std}} \wedge \dots \wedge \Omega_{\on{std}} 
\]
\end{definition}

\begin{remark}[Normalization] Under our convention, the Holmes-Thompson volume of a Riemannian metric differs from the Riemannian volume by the volume of cotangent fibers. Thus
\[\volHT(S^2, G_{\on{std}}) = 4\pi^2 \qquad\text{and}\qquad \volHT(\RP^2, G_{\on{std}}) = 2\pi^2\]
where $G_{\on{std}}$ denotes the standard round Riemannian metrics on $S^2$ and $\mathbb{RP}^2$ in these formulas. \end{remark}

\begin{definition}[Reversible] A Finsler metric $F$ is \emph{reversible} if the fiberwise norm is symmetric on each tangent space. \end{definition}

\begin{definition}[Reversibilization] The \emph{reversibilization} $\bar{F}$ of a Finsler metric $F$ on a manifold $M$ is defined by the following formula.
\[
\bar{F}_x(v) = \frac{1}{2}(F_x(v) + F_x(-v))
\]
\end{definition} 

We will need a few elementary results about Finsler metrics. First, we have the following variant of the John ellipsoid theorem for reversible Finsler metrics. 

\begin{thm}[John Metric] \label{thm:johns_metric} Let $F$ be a reversible Finsler metric on a manifold $M$. Then there is a Riemannian metric $G$ such that
\[
G \le F^2 \le n \cdot G \qquad\text{where $n = \on{dim}(M)$.}
\]
\end{thm}

\noindent Second, we have the following estimate for the volume of the reversibilization, which is a variant of the Rogers-Shephard inequality for Finsler metrics.  

\begin{lemma}[Reversibilization Volume] \label{lem:reversibilization_volume} The Holmes-Thompson volume of the reversibilization of a continuous Finsler metric $F$ on a compact manifold $M$ satisfies
\[
\volHT(M,\bar{F}) \le \frac{(2n)!}{2^n(n!)^2} \cdot \volHT(M,F).
\]
with equality if and only if the unit disk $D_F M$ of $F$ is a simplex in each fiber. \end{lemma}

\begin{proof} We may assume without loss of generality that $M$ is a domain in $\R^n$. The codisk bundles of $F$ and $\bar{F}$ can then be viewed as two families of convex subsets
\[
K_x \subset \R^n \quad\text{and}\quad \bar{K}_x \subset \R^n \qquad\text{parametrized by $x \in M$}
\]
The Holmes-Thompson volumes of $M$ with respect to $F$ and $\bar{F}$ are given by the integrals of $\on{vol}(K_x)$ and $\on{vol}(\bar{K}_x)$ over $M$, respectively, where $\on{vol}(-)$ denotes the Euclidean volume. On the other hand, by the symmetrization construction, one may check that $\bar{K}$ and $K$ are related by
\[
\bar{K}_x = (K_x - K_x)/2
\]
We can then apply the Rogers-Shephard inequality (cf. Schneider \cite[Thm 10.1.4]{schneider1993convex}) to show that
\[
\on{vol}(\bar{K}_x) = \frac{1}{2^n}\on{vol}(K_x - K_x) \le \frac{(2n)!}{2^n(n!)^2} \cdot \on{vol}(K_x)
\]
Finally, the Rogers-Shephard inequality achieves equality if and only if the convex body is a simplex \cite{schneider1993convex}. Therefore, equality can occur in the above inequality for every $x$ if and only if the dual ball ${K}_x$ is a simplex at each $x$. Since a convex body is a simplex if and only if its dual is a simplex, this yields the result. \end{proof}

\subsection{Transverse Isotopies Of Cotangent Lifts} \label{subsec:linking} Recall that a sub-manifold $S$ in a contact manifold $(Y,\xi)$ is \emph{transverse} if the tangent space of $S$ is everywhere transverse to $\xi$, and an isotopy $S_t \subset Y$ is \emph{transverse} if $S_t$ is transverse for each $t$. 

\vspace{3pt}

We conclude this section with a discussion of transverse isotopies of cotangent lifts of curves in surfaces. We first must recall the notion of the Legendre transform.

\begin{definition}[Legendre Transform] The \emph{Legendre transform} associated to a smooth, strongly convex asymmetric norm $F$ on a vector space $V$ is the diffeomorphism
\[
\mathcal{L}:S_F \to S^*_F
\]
from the unit sphere $S_F = \partial D_F$ of $F$ to the dual unit sphere of $F$ that maps a unit vector $v$ of $F$ to the unique unit dual vector $\mathcal{L}(v)$ whose kernel is the tangent space to $S_F$ at the point $v$ and satisfies  $\mathcal{L}(v)(v)=1$. \end{definition}

For a manifold $M$ with a Finsler metric $F$, denote the unit tangent bundle by $S_F M=\partial D_F M$. Likewise, let $F^*$ denote the dual metric and let
\[S_F^*M =\partial D_F^*M\]
be the unit cotangent bundle with respect to the dual metric. The fiberwise Legendre transform then determines a smooth bundle isomorphism
\[
\mathcal{L}:S_FM \to S^*_FM \qquad\text{mapping $v$ in $T_xM$ to $\mathcal{L}_x(v)$ in $T^*_xM$.}
\]

\begin{definition}[Tangent/Cotangent Lifts]\label{def:cotangent_lift} The \emph{(unit) tangent lift} of an immersed loop $\gamma:S^1\rightarrow M$ in a Finsler manifold $(M,F)$ is the loop in $S_FM$ defined as follows.
\begin{equation}
    \hat \gamma:S^1\rightarrow S_FM \qquad\text{given by}\qquad \hat\gamma(t) = \left(\gamma(t)\ ,\  \frac{\frac{d\gamma}{dt}(t)}{F_{\gamma(t)}(\frac{d\gamma}{dt}(t))}\right).
\end{equation}
The \emph{(unit) cotangent lift} of $\gamma$ to $S_F^* M$ is the composition of the tangent lift $\hat{\gamma}:S^1 \to S_FM$ with the Legendre transform $\mathcal{L}:S_FM \to S^*_FM$. 
\end{definition}

\begin{lemma} \label{lem:cotangent_lifts} 
    Cotangent lifts of immersed loops are transverse to the contact structure on $S^*M$.
\end{lemma}
\begin{proof}
    Recall that the standard contact form $\alpha$ on $S_F^*M$ is defined so that, at a point $z  = (x,p)$ in $S^*_FM$ and a tangent vector $w$ to $S^*_FM$ at the point $(x,p)$, we have
    \[
    \alpha_z(w) = p(d\pi_z(w))\qquad \text{where} \qquad \pi:S_F^*M\rightarrow M\text{ is the bundle projection.}
    \]
    Now fix a parameter $t$, and for simplicity fix the following notation.
    \[ z = (x,p) = \widetilde{\gamma}(t) \qquad w = \frac{d\tilde{\gamma}}{dt}(t) \qquad v = d\pi_z(w) = \frac{d\gamma}{dt}(t)\]
    Note that the final identity follows from the fact that $\gamma = \pi \circ \widetilde{\gamma}$ by construction of the cotangent lift. Moreover, $p$ is given by the fiberwise Legendre transform
    \[
    p = \mathcal{L}_x(\frac{v}{F_x(v)})
    \]
    In particular, $p$ applied to the vector $v/F_x(v)$ is one. Finally, to show that $\widetilde{\gamma}$ is transverse to $\xi$ at $t$, we must show that $\alpha_z(w) > 0$. We simply compute that
    \[
    \alpha_z(w) = p(d\pi_z(w)) = p(v) = F_x(v) \cdot p(\frac{v}{F_x(v)}) = F_x(v) > 0
    \]
    Note that $F_x(v)$ is positive since $\gamma$ is an immersion, and thus $v$ is non-zero.
\end{proof}
\begin{cor}\label{cor:tangent_lifts_trans_isot}
    The cotangent lifts of  embedded loops in $S^2$ are all transversely isotopic.
\end{cor}
\begin{proof}
     Given any two embedded loops $\gamma_0$ and $\gamma_1$ on $S^2$, they are isotopic through embedded loops $\gamma_t$. The cotangent lifts $\tilde \gamma_t$ are all transverse to the contact distribution and therefore $\tilde \gamma_0$ and $\tilde \gamma_1$ are transversely isotopic.
\end{proof}

Cotangent lifts are independent of the metric, up to a transverse isotopy.
\begin{lemma}\label{lem:cot_lift_idep_F}
    Let $M$ be a two dimensional manifold with a Finsler metric $F$, and let   $\gamma_1,\dots,\gamma_n$ be a collection of oriented embedded  nowhere tangent curves in $M$.  The link formed by their cotangent lifts $L:=\tilde\gamma_1\cup\cdots\cup\tilde \gamma_n\subset S_F^*M$ with respect to a Finsler metric $F$, is independent of $F$ up to transverse isotopy.
\end{lemma}
\begin{proof}
    Let $F'$ be another Finsler metric on $M$ and consider the isotopy of Finsler metrics
    \[
    F_t:= \sqrt{(1-t)F^2+t\|\cdot\|_{F'}^2}.
    \]
    The cotangent lifts of $\gamma_1,\dots,\gamma_n$ with respect to $F_t$ give a transverse isotopy between the $F$-lifts and the $F'$-lifts. Here we identify $S_{F_t}^*M$ with $S_{F}^*M$ via radial projection in each fiber. 
\end{proof}

To study transverse link types of $n$ minimizing Reeb orbits, it will be useful to understand isotopies without tangencies of certain loop configurations in $\RP^2$.
\begin{definition}[Isotopy Without Tangencies] Two collections of embedded, transversely intersecting loops $\{\gamma_i^0\}_i, \{\gamma_i^1\}_i$ are \emph{isotopic without tangencies} if there is a smooth isotopy through immersions
    \[
    \phi:[0,1]\times \cup_i S^1\rightarrow M, \quad \text{such that}\quad \phi_0(\cup_i S^1) = \{\gamma_i^0\}_i,\ \phi_1(\cup_i S^1) = \{\gamma_i^1\}_i 
    \]
    and for all $t$, all intersections (mutual or self) of the loops obtained by restricting  $\phi(t,-)$ to different $S^1$ are  transverse.
\end{definition}
\begin{remark} \label{rmk:lifting_isotopies}
    Note that if $\{\gamma_i^0\}_i, \{\gamma_i^1\}_i$ are isotopic without tangencies then their cotangent lifts $\{\tilde \gamma_i^0\}_i, \{\tilde\gamma_i^1\}_i$ are transversely isotopic links in $S^*_FM$. Moreover, isotopy without tangencies is independent of the orientation of the curves, and therefore the links obtained by cotangent lifts with respect to all possible orientations
    \[
    \tilde \gamma_1^j \cup -\tilde \gamma_1^j \cup\cdots\cup \tilde \gamma_n^j \cup -\tilde \gamma_n^j, \qquad j=0,1
    \]
    are also transversely isotopic.
\end{remark}
The following proposition is  \cite[Proposition 6.4]{ruberman2019topological}, stated in our notations and wording.
\begin{prop}[{\cite[Proposition 6.4]{ruberman2019topological}}]\label{prop:config_immersely_isotopic}
    Let $\gamma_1,\dots , \gamma_n$ be embedded loops in $\RP^2$, such that 
    \begin{equation}\label{eq:homotopy+intersection_condition}
        [\gamma_i]=[\RP^1]\text{ in }\pi_1(\RP^2) \quad \text{and}\quad  |\gamma_i\pitchfork \gamma_j| = 1 \quad \text{for all }i\neq j.
    \end{equation}
    Then, the union $\gamma_1\cup\cdots \cup \gamma_n$ is isotopic without tangencies to the union of $n$ great circles in $\RP^2$.
\end{prop}
Roughly speaking, the idea is that if a tangency in the isotopy cannot be avoided by perturbations, then the curves must have at least 3 intersection points (of which two form an algebraically canceling pair, that will lead to the tangency).

\begin{cor}\label{cor:loops_in_RP2_are_great_circles}
    Let $\gamma_1,\dots , \gamma_n$ be embedded loops in $\RP^2$ satisfying (\ref{eq:homotopy+intersection_condition}).
    Let $\eta_i$ be the lift of $\gamma_i$ to $S^2$. Then the union of their cotangent lifts $\cup_i \tilde \eta_i$ is transversely isotopic in $S^*_F S^2$ to a union of cotangent lifts of great circles.
\end{cor}
\begin{proof}
    By Proposition~\ref{prop:config_immersely_isotopic}, $\{\gamma_i\}$ are isotopic without tangencies to a union of great circles in $\RP^2$. Lifting this isotopy to $S^2$, the lifts $\{\eta_i\}$ are isotopic without tangencies to a union of great circles in $S^2$. Since isotopy without tangencies lifts to a transverse isotopy of the cotangent lifts, we see that the union of $\{\tilde \eta_i\}$ is transversely isotopic to the union of cotangent lifts of great circles.
\end{proof}

\section{Systolic Geometry} \label{sec:systolic_geometry} 

In this section, we review various background results in systolic geometry for Riemannian metrics and Finsler metrics.  These results will serve as key tools in the applications discussed in the introduction.


\subsection{Metric Systoles} We start by defining the metric systole, which is the most elementary invariant in systolic geometry.

\begin{definition}[Metric Systole] The \emph{metric systole} of a Finsler manifold  $(M,F)$ is the infimum over all non-contractible rectifiable loops in $M$ of the length of that loop.
\[
\on{sys}(M,F) := \on{inf}\big\{\on{length}_F(\Gamma) \; : \; \text{rectifiable non-contractible loops }\Gamma \subset M\big\}
\]
\end{definition}

The following scaling and monotonicity properties of the metric systole follow immediately from the analogous properties of the length.

\begin{lemma}[Systole Scaling] \label{lem:sys_scaling} The metric systole satisfies the following property with respect to scaling.
\[
\on{sys}(M,\lambda F) = \lambda \cdot \on{sys}(M,F) \qquad\text{for any constant $\lambda > 0$.}
\]
\end{lemma}

\begin{lemma}[Systole Monotonicity] \label{lem:sys_monotonicity}  The metric systole is monotonic in the sense that
\[
\on{sys}(M,F) \le \on{sys}(M,F') \qquad\text{if $F$ and $F'$ are Finsler metrics on $M$ with $F \le F'$.}
\]
\end{lemma}

\noindent As a consequence of these results, the metric systole is continuous in the Lipschitz topology. 

\begin{lemma}[Systole Continuity] \label{lem:sys_continuity} The metric systole is continuous in the Lipschitz topology on Finsler metrics on any compact manifold.
\end{lemma}

\begin{proof} Suppose that $F_i$ is a sequence of Finsler metrics converging to a Finsler metric $F$ in the Lipschitz topology. By definition of the topology, there is a sequence $\epsilon_i$ with $\epsilon_i \to 0$ such that
\[
(1 - \epsilon_i) \cdot F\le F_i \le (1+\epsilon_i) \cdot F
\]
It then follows from Lemma \ref{lem:sys_scaling} and Lemma \ref{lem:sys_monotonicity} that
\[
(1 - \epsilon_i) \cdot \on{sys}(M,F) \le \on{sys}(M,F_i)\le (1 + \epsilon_i) \cdot \on{sys}(M,F)
\]
Therefore the metric systoles of $F_i$ converge to the metric systole of $F$.  \end{proof}

\noindent Next, it is a standard result that there is a closed loop that has length equal to the metric systole, which is necessarily a geodesic (cf. \cite[Thm 8.7.1]{BaoChernShen2000}). We refer to such a loop as a \emph{systolic geodesic}. 

\begin{thm}[Systole Spectrality] \label{thm:sys_spectral}  For any Finsler metric $F$ on a closed manifold $M$ with positive finite systole, there exists a closed non-contractible geodesic $\Gamma$ on $M$ with respect to $F$ such that
\[\on{length}_F(\Gamma) = \on{sys}(M,F)\]
\end{thm}

\noindent For our purposes, it will be important that systolic geodesics are embedded.

\begin{lemma}[Embedded Systoles] \label{lem:embedded_systoles} Any systolic geodesic $\Gamma$ of a Finsler metric $F$ is embedded.
\end{lemma}

\begin{proof} Let $\Gamma$ be a systolic geodesic. Suppose that $\Gamma$ intersects itself at a point $x$. Then we can decompose $\Gamma$ as a path composition $\Gamma_1 \star \Gamma_2$ where $\Gamma_1$ and $\Gamma_2$ are closed loops based at $x$. The length of both curves is less than the length of $\Gamma$, and one of these loops must be non-contractible since $\Gamma$ is non-contractible. It follows that there is a non-contractible loop of smaller length than $\Gamma$, which is a contradiction. Therefore $\Gamma$ is embedded.\end{proof}

\subsection{Systoles On Projective Planes} In this paper, a central role will be played by systolic inequalities that relate the systole and volume for Finsler metrics on the projective plane. In the reversible case, there is a sharp systolic inequality generalizing the Pu inequality \cite{pu1952some} due to Ivanov \cite{ivanov2002two}. 

\begin{thm}[Ivanov-Pu, \cite{ivanov2002two}] \label{thm:ivanov_pu} Let $F$ be a continuous reversible Finsler metric on $\mathbb{RP}^2$. Then
\begin{equation} \label{eq:ivanov_pu}
\on{sys}(\mathbb{RP}^2,F)^2 \le \frac{1}{2} \cdot \volHT(\mathbb{RP}^2,F)
\end{equation}
Moreover, a strongly convex and smooth Finsler metric $F$ achieves equality if and only if $F$ is Zoll. \end{thm}

\noindent In the case of non-reversible Finsler metrics, a non-optimal systolic inequality can be derived from the Ivanov-Pu inequality via the following symmetrization construction.

\begin{lemma}[Reversibilization Systole] \label{lem:rev_systole} The systole of the reversibilization of a Finsler metric $F$ on a manifold $M$ satisfies the following inequality
\[
\on{sys}(M,F) \le \on{sys}(M,\bar{F})
\]\end{lemma}

\begin{proof} Note that if $\Gamma$ is any non-contractible loop and $-\Gamma$ is the same curve with the orientation reversed, then
\[
2 \cdot \on{sys}(M,F) \le \on{length}_F(\Gamma) + \on{length}_F(-\Gamma) = 2 \cdot \on{length}_{\bar{F}}(\Gamma)
\]Taking the infimum over all non-contractible loops $\Gamma$ yields the desired systolic bound. \end{proof}

\begin{prop}[Non-Sharp Pu] \label{prop:non_reversible_Pu_sota} Let $F$ be a Finsler metric on the projective plane $\mathbb{RP}^2$. Then
\[
\on{sys}(\mathbb{RP}^2,F)^2 \le \frac{3}{4} \cdot \volHT(\mathbb{RP}^2,F)
\]
Moreover, any Finsler metric achieving equality must have unit cotangent fibers given by triangles.\end{prop}

\begin{proof} We simply apply Theorem \ref{thm:ivanov_pu}, Lemma \ref{lem:rev_systole} and Lemma \ref{lem:reversibilization_volume} to see that
\[
\on{sys}(\mathbb{RP}^2,F)^2 \le \on{sys}(\mathbb{RP}^2,\bar{F})^2 \le \frac{1}{2} \cdot \volHT(\mathbb{RP}^2,\bar{F}) \le \frac{1}{2} \cdot \frac{3}{2} \cdot \volHT(\mathbb{RP}^2,F). 
\]
Moreover, equality can only be achieved if the last inequality is an equality, which occurs if and only if the fibers of the unit disk bundle of $F$ are triangles by Lemma \ref{lem:reversibilization_volume}. Since the dual of a triangle is again a triangle, this occurs only if the fibers of $D^*_F\mathbb{RP}^2$ are triangles. \end{proof}

\noindent Identifying the optimal systolic inequalities for non-reversible Finsler metrics on the projective plane remains an open problem. In particular, it is subject to the following conjecture.

\begin{conjecture}[Non-Reversible Pu \cite{alonso2024survey}]\label{conj:nonreversible_pu} Let $F$ be a (non-reversible) Finsler metric on $\mathbb{RP}^2$. Then
\[
\on{sys}(\mathbb{RP}^2,F)^2 \le \frac{1}{2} \cdot \volHT(\mathbb{RP}^2,F)
\]
Moreover, a strongly convex and smooth Finsler metric $F$ achieves equality if and only if $F$ is Zoll. \end{conjecture}

Pairs of systolic geodesics on a reversible Finsler projective plane also satisfy the following useful intersection property. 

\begin{lemma}[Systolic Intersections] \label{lem:systolic_intersections} Any pair of distinct systolic geodesics for a reversible Finsler metric $F$ on $\mathbb{RP}^2$ intersect exactly once.
\end{lemma}

\begin{proof} Let $\Gamma_1$ and $\Gamma_2$ be two systolic geodesics. The homology classes $[\Gamma_1]$ and $[\Gamma_2]$ each generate the first homology group of $\mathbb{RP}^2$. Therefore their mod 2 intersection number is one, and the curves $\Gamma_1$ and $\Gamma_2$ must intersect at least once. 

\vspace{3pt}

Suppose that $\Gamma_1 \cap \Gamma_2$ contains two distinct points $x$ and $y$. Then we can decompose $\Gamma_1$ and $\Gamma_2$ into path compositions
\[
\Gamma_i = \Gamma_i^x \star \Gamma_i^y \qquad\text{for $i = 1,2$.}
\]
Here $\Gamma_i^x$ is a path from $x$ to $y$ and $\Gamma_i^y$ is a path from $y$ to $x$ for each $i$. We now consider the closed piecewise smooth curves given by
\[
\widetilde{\Gamma}_{xx} = \Gamma_1^x \star -\Gamma^x_2 \qquad \widetilde{\Gamma}_{xy} = \Gamma_1^x \star \Gamma_2^y \qquad \widetilde{\Gamma}_{yy} = \Gamma_1^y \star -\Gamma_2^y \qquad  \widetilde{\Gamma}_{yx} = -\Gamma_1^y \star -\Gamma_2^x
\]
\begin{figure}
    \centering
    \includegraphics[width=\linewidth]{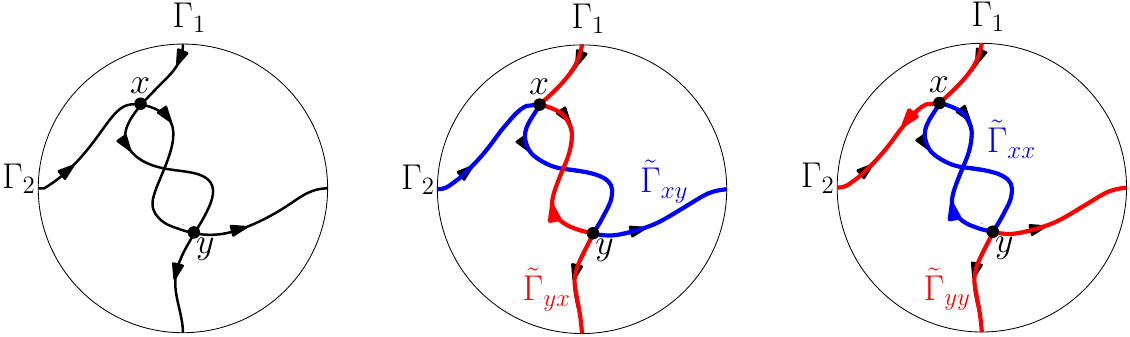}
    \caption{An illustration of geodesics $\Gamma_1, \Gamma_2$ intersecting at more than one point, and decompositions of their union into different pairs of loops.}
    \label{fig:intersecting_systoles}
\end{figure}
See Figure~\ref{fig:intersecting_systoles}.
Here $-\Gamma_i$ denotes the curve $\Gamma_i$ with the reverse parametrization. Note that we have the following formulas in homology.
\[
[\widetilde{\Gamma}_{xx}] - [\widetilde{\Gamma}_{xy}] = [\Gamma_1]
\]
This implies that one of the curves $\widetilde{\Gamma}_{xx}$ and $\widetilde{\Gamma}_{xy}$ must be non-contractible.  Without loss of generality, assume $\widetilde{\Gamma}_{xy}$ is non-contractible, the proof of the other case is analogous. Then,
\[
[\widetilde{\Gamma}_{xy}] + [\widetilde{\Gamma}_{yx}] =[\Gamma_1]+[\Gamma_2] = 0 \qquad\text{and}\qquad \on{length}_F(\widetilde{\Gamma}_{xy}) + \on{length}_F(\widetilde{\Gamma}_{yx}) = 2 \cdot \on{sys}(\mathbb{RP}^2,F)
\]
It follows that both $\widetilde{\Gamma}_{xy}$ and $\widetilde{\Gamma}_{yx}$ are non-contractible with minimal length. They must therefore be smooth closed geodesics. However, this implies that the unit tangent vector of $\widetilde{\Gamma}_{xy}$ is well-defined at $x$, and equal to the tangent vector of both $\Gamma_1$ and $\Gamma_2$ at $x$. Two geodesics with the same tangent vector at a point coincide, so this is a contradiction in this case.  Note that we must use the reversibility of $F$ to conclude that
\[
\on{length}_F(\widetilde{\Gamma}_{xy}) + \on{length}_F(\widetilde{\Gamma}_{yx}) = \on{length}_F(\Gamma_1) + \on{length}_F(\Gamma_2)  = 2 \cdot \on{sys}(\mathbb{RP}^2,F)
\]
Otherwise, this case is identical to the first case. This concludes the proof. \end{proof}

\subsection{Shortest Geodesics} In this paper, we will also be interested in length-area inequalities for Finsler metrics on the sphere. Since the sphere is simply connected, we must replace the systole with the shortest geodesic length.

\begin{definition}[Shortest Geodesic] The \emph{shortest geodesic length} of a Finsler metric $F$ on a closed manifold $M$ is the minimal length of a non-constant closed geodesic on $M$.
\[
\on{sg}(M,F) := \on{min}\big\{\on{length}_F(\Gamma) \; : \; \text{ non-constant closed $F$-geodesic }\Gamma \subset M\big\}
\]
\end{definition}

The shortest geodesic length does not satisfy the same monotonicity and continuity properties as the systole. However, one can still obtain a length-area bound relating the shortest geodesic with the area using an auxiliary invariant called the theta width (cf. \cite[Thm 1.3]{balacheff2015measurements}). 

\begin{thm}[Balacheff] \label{thm:shortest_geodesic_ineq} Let $F$ be a (possibly non-reversible) Finsler metric on the sphere $S^2$. Then
\[
\on{sg}(S^2,F)^2 \le C \cdot \volHT(S^2,F) \qquad\text{where}\qquad C \le (8 + \frac{11\sqrt{3}}{2})^2
\]
\end{thm}

\begin{remark} Note that the constant in Theorem \ref{thm:shortest_geodesic_ineq} differs from the constant appearing in \cite[Thm 1.3]{balacheff2015measurements} by a factor of $\pi$ due to our normalization of the Holmes-Thompson volume. 
\end{remark}

\section{The Quaternion Cover} \label{sec:quaternion_covering}

In this section, we construct the quaternion covering map from the boundary of any centrally symmetric domain in the quaternions that is convex along $J$-complex lines to the unit cosphere bundle of the 2-sphere (Theorem \ref{thm:quaternion_cover} and Corollary~\ref{cor:quaternion_cover}). We then use this construction to prove generalizations of Theorems \ref{thm:intro_Q_Viterbo} and \ref{thm:intro_I_invariant_systolic} from the introduction. 

\subsection{Statement Of Covering Result} We start by stating our main covering theorem. The following notation for various structures on the quaternions will be useful.

\begin{notation}[Quaternion Structures] The quaternions can be divided into a direct sum of real and imaginary subspaces
\[
\on{Re} \HH = \R = \on{span}(1) \qquad\text{and}\qquad \im \HH = \on{span}(I,J,K)
\]
We denote the real part, imaginary part and conjugate of a quaternion $z$ by
\[
\on{Re}(z) \qquad \im(z) \qquad \bar{z} = \on{Re}(z) - \im(z)
\]
The quaternions also contain three different copies of the complex numbers determined by a choice of complex structure. 
\[
\C_I = \on{span}(1,I) \qquad \C_J = \on{span}(1,J)  \qquad \C_K = \on{span}(1,K)
\]
The standard inner product, symplectic form and Liouville form on the quaternions are given by
\[
\langle w,z\rangle = \on{Re}(w\bar{z}) \qquad \Omega(w,z) = \on{Re}(wI\bar{z}) \qquad \lambda_{\HH,z}(v) = \frac{1}{2}\on{Re}(zI\bar{v})
\]
We will also need to use the standard Liouville form on $T^*\HH = \HH \times \HH$. We will denote this Liouville form by $\lambda_{\HH \times \HH}$, which can be written explicitly as
\begin{equation}\label{eq:theta_def} \lambda_{\HH \times \HH,w,z}(u,v) = \langle z,u\rangle\end{equation} \end{notation}

We will also need the following action of the quaternion group on the cotangent bundle of the sphere. Note that in the following definition, we regard $T^*S^2$ as a subset of $T^*\R^3 = \R^3 \times \R^3$. 

\begin{definition}[Cotangent Action] The \emph{standard $Q$-action} of the quaternion group on the cotangent bundle $T^*S^2$ is the unique smooth action such that
\[
I \cdot (x,v) = (-x,-v) \qquad J \cdot (x,v) = (x,-v) \qquad K \cdot (x,v) = (-x,v)
\]
\end{definition}

\begin{remark} Note that a Finsler metric $F$ on $S^2$ descends to the projective plane $\mathbb{RP}^2$ if and only if it is $I$-invariant. Similarly, a Finsler metric is reversible if and only if it is $J$-invariant.
\end{remark}

Our main covering result can now be stated using the notation and definitions above as follows. 

\begin{thm}[Quaternion Covering] \label{thm:quaternion_cover} There is a two-sheeted, $Q$-equivariant exact symplectic covering
\[
\pi:(\HH^\times,\lambda_{\HH}) \to (T^*S^2 - S^2,\lambda_{T^*S^2})
\]
such that, for any  centrally symmetric smooth domain $X \subset \HH$ that is convex along $J$-complex lines, the image $\pi(X - 0)$ is given by
\[
D^*_FS^2 - S^2 \qquad\text{for a unique smooth Finsler metric $F$ on $S^2$}.
\]
Moreover, the unit disk bundle $D^*_FS^2$ has squared fibers in the sense of Definition \ref{def:squared_domain} and the ball $B \subset \HH$ corresponds to the multiple $\frac{1}{16} G_{\on{std}}$ of the standard round Riemannian metric  on $S^2$. \end{thm}

\begin{remark} We will refer to the metric $F$ in Theorem \ref{thm:quaternion_cover} as the \emph{corresponding Finsler metric} to $X$.
\end{remark}

\begin{remark}[Lipschitz Metrics] \label{rmk:lipschitz_metrics} For a general centrally symmetric convex body $X \subset \HH$, there is an associated Lipschitz Finsler metric $F$ defined by the property that $D^*_FS^2 - S^2 = \pi(X - 0)$. In this setting, $F$ is not necessarily strongly convex, but $F^*$ is squared (and in particular strongly convex) fiberwise.
\end{remark}

\noindent In our applications, it will be more convenient to use the following version of Theorem \ref{thm:quaternion_cover} stated using the cosphere bundle.

\begin{cor}[Cosphere Covering] \label{cor:quaternion_cover} Let \(X \subset \mathbb{H}\) be a smooth centrally symmetric domain that is convex along $J$-complex lines. Then there is a Finsler metric \(F\) on \(S^2\) and a two-sheeted covering map
\[ \pi : \partial X \longrightarrow S_F^*S^2,  \qquad\text{such that}\qquad \pi^*\lambda_{T^*S^2} = \lambda_{\HH}|_{\partial X}.\]
Moreover, if $X$ is invariant under right-multiplication by $J$ then $F$ is reversible and if $X$ is invariant under right-multiplication by $I$, then $F$ descends to a Finsler metric on $\mathbb{RP}^2$. \end{cor}

\begin{remark}[Metrics-To-Domains] Unfortunately, the converse of Theorem \ref{thm:quaternion_cover} is not true. There are Finsler metrics that do not lift to centrally symmetric domains that are convex along $J$-complex lines. Specifically, the fiberwise convexity may fail. However, if a given Finsler metric lifts to a corresponding domain $X$, then any metric that is conformally equivalent to it also lifts to a domain. In particular, we have the following proposition.
\end{remark}
\begin{cor}[Lifting Metrics] Let $F$ be a Finsler metric that is conformally equivalent to the round metric $g_{\on{std}}$. Then $F$ corresponds to a centrally symmetric domain $X$ convex along $J$-complex lines.
\end{cor}
Note that every Riemannian metric is diffeomorphic to one that  is conformally equivalent to the round metric $g_{\on{std}}$.
\subsection{Construction Of Covering} Let us now construct the covering map described in Theorem \ref{thm:quaternion_cover}. We start by introducing some useful maps. Note that $\im \HH$ is canonically identified with $\mathbb{R}^3$. There is thus a canonical inclusion of $S^2$ into $\HH$ and the differential induces an inclusion
\[
TS^2 \to T\mathbb{H} = \mathbb{H} \times \mathbb{H} \qquad\text{where}\qquad T_z\HH = z \times \HH
\]
By using the identification $T^*S^2 = TS^2$ via the standard round metric, this induces an inclusion
\begin{equation} \label{eq:TS2_inclusion}
\iota:T^*S^2 \to \mathbb{H} \times \mathbb{H}\end{equation}
We denote the image of the complement of the zero section in $T^*S^2$ under $\iota$ by
\begin{equation}
W = \iota(T^*S^2 - S^2)
\end{equation}
We will also need to consider the following map from the group of units within the quaternions. 
\begin{equation}
\psi:\mathbb{H}^\times \to \mathbb{H} \times \mathbb{H} \qquad\text{given by}\qquad \psi(z) =  \Big(\frac{zJ\bar{z}}{|z|^2}, \frac{1}{4} zK\bar{z}\Big) 
\end{equation}
We denote the first and second components of this map by $\psi_1$ and $\psi_2$ respectively. 

\vspace{3pt}

We next establish various key properties of the maps $\iota$ and $\psi$. With these properties in hand, the proof of Theorem \ref{thm:quaternion_cover} will be straightforward. 

\begin{figure}[h!]
    \centering
    \includegraphics[width=.4\linewidth]{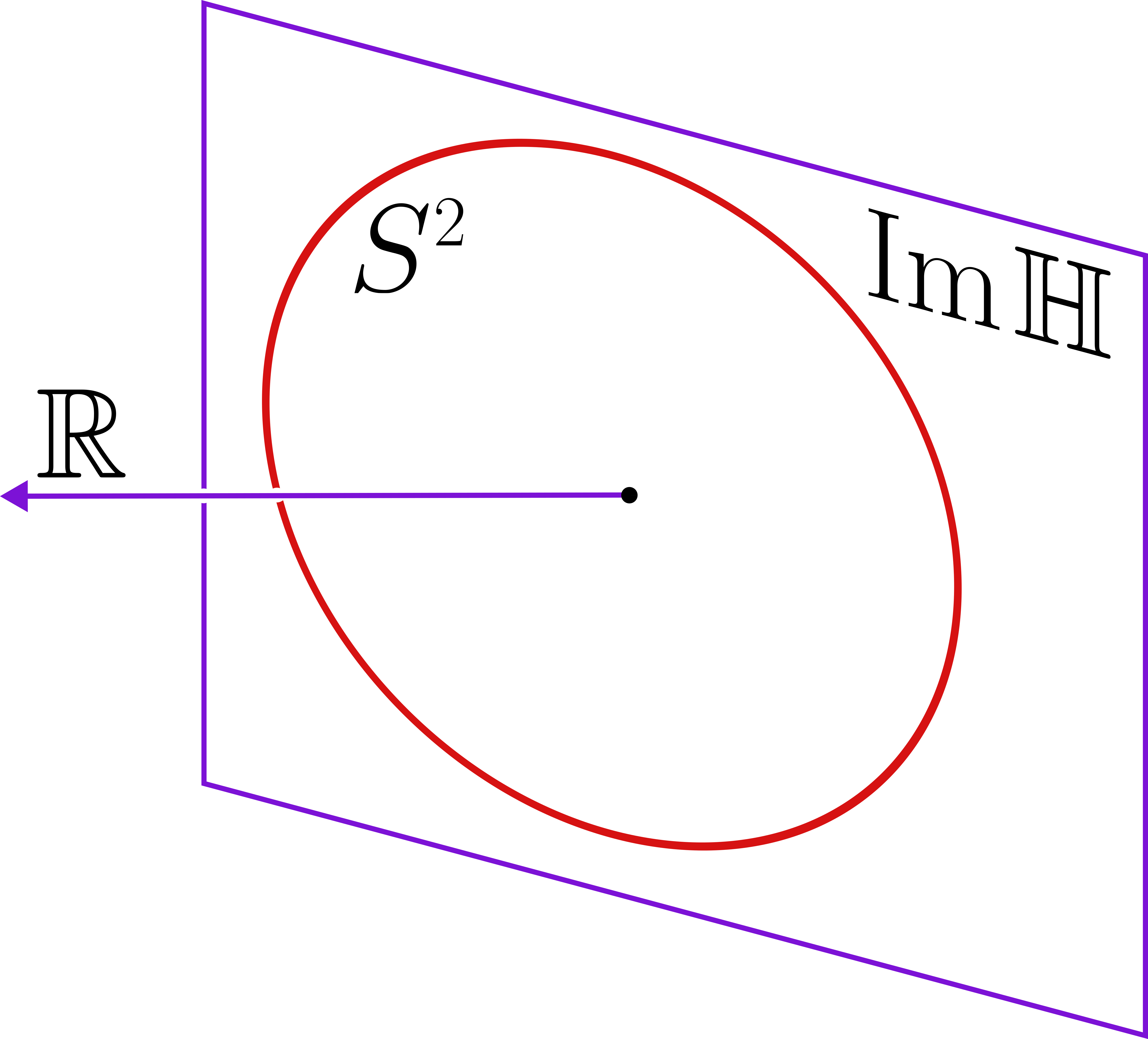}
    \caption{A cartoon of the unit sphere in the hyperplane spanned by $I,J$ and $K$ in $\mathbb{H}$.}
    \label{fig:sphere_in_H}
\end{figure}

\begin{lemma}[Imaginary And Orthogonal] \label{lem:image_j}For any $z$ in $\HH^\times$, the images $\psi_1(z)$ and $\psi_2(z)$ are imaginary and mutually orthogonal. Moreover, the product of $\psi_1(z)$ and $\psi_2(z)$ is given by
\[\psi_1(z)\psi_2(z) = \frac{1}{4}zI\bar{z}\]
\end{lemma}

\begin{proof} Note that an element $w$ in $\HH$ is imaginary if and only if it is the negative of its conjugate. This holds for the products $zI\bar{z}$ since
\[
\overline{zI\bar{z}} = \bar{\bar{z}}\bar{I}\bar{z} = -zI\bar{z} 
\]
The same calculation applies to $zJ\bar{z}$ and $zK\bar{z}$. It follows that $\psi_1(z)$ and $\psi_2(z)$ are imaginary, or equivalently contained in $\on{Im}\HH$. Next, we compute that
\[
\psi_1(z)\psi_2(z) = \frac{1}{4|z|^2} \cdot zJ\bar{z}zK\bar{z}  = \frac{1}{4} \cdot zJK\bar{z} = \frac{1}{4} \cdot zI\bar{z}
\]
Finally, we note that $\psi_1(z)$ is orthogonal to $\psi_2(z)$ since the inner product of $\psi_1(z)$ and $\psi_2(z)$ is the real part of $\psi_1(z)\overline{\psi_2(z)} = -\psi_1(z)\psi_2(z)$, which is zero.
\end{proof}

\begin{lemma}[Right Action] \label{lem:right_multiplication}The map $\psi$ has the following properties under the right multiplication.
\begin{equation} \label{eq:C_J_action}
\psi_1(zu) = \psi_1(z) \qquad \text{and}\qquad \psi_2(zu) = \frac{1}{4}zu^2K\bar{z}  \qquad\text{for any $u \in \C_J - 0$}
\end{equation}
\begin{equation} \label{eq:Q_action} 
\psi(-z) = \psi(z) \qquad \psi(zI) = -\psi(z) \qquad\text{and}\qquad \psi(zK) = (-\psi_1(z),\psi_2(z))
\end{equation}
\end{lemma}

\begin{proof} We prove these identities by direct computation. For the identities (\ref{eq:C_J_action}), note that $uJ = Ju$ and $uK = K\bar{u}$ for any element $u$ in $\C_J$. Therefore
\[
\psi_1(zu) = \frac{zuJ\bar{u}\bar{z}}{|u|^2|z|^2} =  \frac{zu\bar{u}J\bar{z}}{|u|^2|z|^2} = \frac{|u|^2}{|u|^2} \frac{zJ\overline{z}}{|z|^2} = \psi_1(z) \qquad\text{and}\qquad 
\psi_2(zu) = \frac{1}{4}zuK\bar{u}\bar{z} = \frac{1}{4}zu^2K\bar{z}
\]
For the identities (\ref{eq:Q_action}), note that the first identity follows from each of the others. To verify the second and third identities, we compute $\psi_i(zI)$ and $\psi_i(zK)$ for $i = 1,2$. Note that
\[
\psi_1(zI) = \frac{zIJ\overline{zI}}{|z|^2} = \frac{-zIJI\ \overline{z}}{|z|^2} = \frac{zI^2J\ \overline{z}}{|z|^2}  = -\psi_1(z)
\]
A similar calculation shows that $\psi_1(zK) = -\psi_1(z)$ and $\psi_2(zI) = -\psi_2(z)$. On the other hand
\[
\psi_2(zK) =  \frac{zKK\overline{zK}}{4} = \frac{-zK^2K\overline{z}}{4} =  \psi_2(z) \qedhere
\]
\end{proof}

\begin{lemma}[Liouville Form]\label{lem:Liouville_j} The maps $\iota$ and $\psi$ respect the Liouville forms in the sense that 
\[\iota^*\lambda_{\HH \times \HH} = \lambda_{T^*S^2} \qquad\text{and}\qquad \psi^*\lambda_{\HH \times \HH} =\lambda_\HH
\]\end{lemma}

\begin{proof} The first identity follows from the fact that the Liouville form $\lambda_{T^*S^2}$ on $T^*S^2$ is identified with the Liouville form on $TS^2$ given precisely by the formula (\ref{eq:theta_def}) restricted to $TS^2 \subset T\mathbb{H}$. For the second identity, we begin by computing the differential of $\psi_1$ at $z$ in the direction $vz$.
\begin{align*}
T_z\psi_1(vz) &= \frac{d}{dt}\big(\frac{(1 + tv)zJ\overline{(1 + tv)z}}{|z(1 + tv)|^2}\big)|_{t = 0} = \frac{vzJ\bar{z} + zJ\bar{z}\bar{v}}{|z|^2} - 2\langle 1,v\rangle \cdot \frac{zJ\bar{z}}{|z|^2} \\
&= v\psi_1(z) + \psi_1(z)\bar{v} - 2\on{Re}(v) \psi_1(z)
\end{align*}
Since $\psi_1(z)$ and $\psi_2(z)$ are mutually orthogonal and both imaginary by Lemma \ref{lem:image_j}, this implies
\begin{align*}
    (\psi^*\lambda_{\HH \times \HH})_z(vz)&= \langle \psi_2(z),T_z\psi_1(vz)\rangle = \langle \psi_2(z),v\psi_1(z) + \psi_1(z)\bar{v}\rangle  \\
    &= \on{Re}(\overline{\psi_2(z)}v\psi_1(z)) + \on{Re}(\overline{\psi_2(z)}\psi_1(z)\bar{v}))
\end{align*}
Moreover, the real part is invariant under cyclic permutation and conjugation, so that
\[
  (\psi^*\lambda_{\HH \times \HH})_z(vz) = \on{Re}(v\psi_1(z)\overline{\psi_2(z)}) + \on{Re}(v\overline{\psi_1(z)}\psi_2(z))) = -2\on{Re}(v\psi_1(z)\psi_2(z)) 
\]
Note that here we use the fact that $\overline{\psi_1(z)} = -\psi_1(z)$ and $\overline{\psi_2(z)} = -\psi_2(z)$. 
Finally, we apply Lemma \ref{lem:image_j} to find that
\[
  (\psi^*\lambda_{\HH \times \HH})_z(vz)  = -\frac{1}{2}\on{Re}(vzI\bar{z}) = \lambda_{\HH}(vz) \qedhere
\]\end{proof}

\begin{lemma}[Covering Map] \label{lem:covering_map} The map $\psi$ is a two-sheeted covering map
\[\HH^\times \to W\]
Moreover, the $\Z/2$-action given by multiplication  by $\pm 1$ acts freely and transitively on the fiber. \end{lemma}

\begin{proof} By Lemma \ref{lem:image_j}, the image of $\psi$ is contained in $W$. By Lemma \ref{lem:Liouville_j}, the map is locally Liouville and therefore locally symplectic. In particular, the differential is non-singular. Moreover, the map $\psi_2$ is proper as a map from $\HH^\times$ to $\HH^\times$ since $4|\psi_2(z)| = |z|^2$ and thus $\psi$ itself is also proper. It follows from Ehresmann's Theorem that $\psi:\HH^\times \to W$ is a fiber bundle with compact fiber. The $\Z/2$-action acts freely on the fibers by Lemma \ref{lem:right_multiplication}. Finally, to compute the fiber, we consider an element $z$ such that $\psi(z) = (J,K/4)$. Since $4|\psi_2(z)| = |z|^2$, we must have $|z| = 1$ and therefore
\[
J = zJ\bar{z} = zJz^{-1} \qquad K = zK\bar{z} = zKz^{-1}
\]
On the other hand, an element $z$ commutes with a complex structure $I,J$ or $K$ if and only if it lies in the corresponding copy of $\C$. Since $\C_J \cap \C_K = \R = \on{span}(1)$, this implies that $z$ is a unit element of $\R$. Equivalently, the fiber of $(J, K/4)$ consists of two elements $z = \pm 1$. \end{proof}

\begin{lemma}[Fiberwise Convexity] \label{lem:fiber_convex} Let $L \subset \mathbb{H}$ be a $J$-complex line and let $X \subset L$ be a closed, centrally symmetric convex domain. Then
\[
\psi_2(X) \subset \mathbb{H} \qquad\text{is a squared convex domain.}
\]
\end{lemma}

\begin{proof} Fix a non-zero element $z$ in $L$. Then any other element $w$ can be written uniquely as $w = zu$ for the element $u = z^{-1}w$ in $\C_J = \C$. By Lemma \ref{lem:right_multiplication}, we can write $\psi_2|_L$ as the composition
\[
\psi_2|_L = B_z \circ \sigma \circ A_z \qquad\text{where}\qquad A_z(w) = z^{-1}w \qquad \sigma(u) = u^2 \qquad B_z(w) = \frac{1}{4}zwK\bar{z}
\]
In particular, $\psi_2$ is given on $L$ by the squaring map $\sigma:\C \to \C$ composed on the left and the right with a linear isomorphism $A_z:L \to \C_J = \C$ and a linear map $B_z:\C \to \HH$. Since $A_z$ is linear, $A_z(X) \subset \C$ is convex and centrally symmetric. Therefore, by Lemma \ref{lem:square_stays_convex}, the image $\sigma(A_z(X)) \subset \C$ is a squared convex domain. Finally, $\psi_2(X)$ is the image $B_z(\sigma(A_z(X)))$ of a squared convex domain under a linear map, and is therefore squared convex.   \end{proof}

We can now proceed with the proof of the main theorem in this section, Theorem \ref{thm:quaternion_cover}. 
\begin{proof}[Proof of Theorem \ref{thm:quaternion_cover}] We simply take the covering map to be
\[
\pi:\HH^\times \to T^*S^2 - S^2 \qquad\text{given by}\qquad \pi = \iota|_W^{-1} \circ \psi
\]
Lemma \ref{lem:covering_map} implies that $\pi$ is a 2-sheeted covering map. Lemmas \ref{lem:Liouville_j} and \ref{lem:right_multiplication} imply that the map preserves the Liouville forms and is $Q$-equivariant. Finally, Lemma \ref{lem:fiber_convex} implies that, if $X$ is a smooth, centrally symmetric domain that is convex along $J$-complex lines, then $\pi(X - 0) \cup S^2$ is a smooth domain in $T^*S^2$ whose fibers are squared convex domains in the fibers of $T^*S^2$. This implies that $\pi(X - 0) \cup S^2$ is the unit codisk bundle of a smooth 1-homogeneous function $F^*$ on $T^*S^2 - S^2$ that is squared, and thus strongly convex, on the fibers. Thus $F^*$ is dual to a smooth Finsler metric $F$ on $TS^2$, and
\[
\pi(X - 0) = D^*_F S^2 - S^2 \qedhere
\]\end{proof}

\section{Applications} \label{sec:applications} We now discuss the proofs of the applications of Theorem \ref{thm:quaternion_cover} to the systolic inequalities described in the introduction.

\vspace{3pt}

\subsection{Quaternionically Symmetric Domains.} First, by combining Corollary \ref{cor:quaternion_cover} with the Ivanov-Pu inequality, we can prove the following generalizations of Theorem \ref{thm:intro_Q_Viterbo} and Theorem \ref{thm:intro_Q_optimizers}.

\begin{thm}[Generalized $Q$-Symmetric Viterbo] \label{thm:generalized_Q_symetric_Viterbo} Let \(X \subset \mathbb{H}\) be a $Q$-invariant smooth domain that is  convex along $J$-complex lines. Then $X$ satisfies the Viterbo inequality.
\end{thm}

\begin{proof} Let $F$ be the corresponding Finsler metric on $S^2$ in Corollary \ref{cor:quaternion_cover}. Since $X$ is $Q$-invariant, $F$ is invariant under $J$ and $I$. Therefore $F$ descends to a reversible Finsler metric
\[
\bar{F} \qquad\text{on the projective plane}\qquad \mathbb{RP}^2
\]
By composing the covering map $\pi$ from Theorem \ref{thm:quaternion_cover} with the covering map of cosphere bundles induced by the covering $S^2 \to \mathbb{RP}^2$, we also have a strict contact 4-sheeted covering map
\begin{equation} \label{eq:Q_invariant_sys_projection}
\bar{\pi}:\partial X \to S^*_{\bar{F}}\mathbb{RP}^2
\end{equation}
Note that any closed geodesic in $\mathbb{RP}^2$ corresponds to a closed Reeb orbit in the cosphere bundle, which will lift to a collection of closed Reeb orbits in $\partial X$ of total period given by four times the length of the original geodesic. The systole of $\bar{F}$ is the length of the shortest non-contractible geodesic of  $\bar{F}$, and is not smaller than the  shortest closed Reeb orbit on the cosphere bundle. Therefore,
\begin{equation} \label{eq:Q_invariant_sys_bound}
\sys(X) \le 4 \cdot \on{sys}(\mathbb{RP}^2,\bar{F})
\end{equation}
Similarly, the Holmes-Thompson volume of $\mathbb{RP}^2$ with respect to $\bar{F}$ is precisely the contact volume of the cosphere bundle. It follows that 
\begin{equation} \label{eq:Q_invariant_vol_bound}
\on{vol}(X) = \on{vol}(\partial X,\lambda_{\HH}|_{\partial X}) = 4 \cdot \on{vol}(S^*_{\bar{F}}\mathbb{RP}^2,\lambda_{T^*\RP^2}) = 4 \cdot \volHT(\mathbb{RP}^2,\bar{F}) 
\end{equation}
Finally, we apply the Ivanov-Pu inequality in Theorem \ref{thm:ivanov_pu} to see that
\begin{equation} \label{eq:Q_invariant_vol_bound_sandwich}
\sys(X)^2 \le 16 \cdot \sys(\mathbb{RP}^2,\bar{F})^2 \le 8 \cdot \volHT(\mathbb{RP}^2,\bar{F}) = 2 \cdot \on{vol}(X)\end{equation} 
This is equivalent to the Viterbo inequality since $\sys(B)^2/\on{vol}(B) = 2$ for a unit ball $B$ in $\R^4$. \end{proof}

\begin{thm}[Generalized $Q$-Symmetric Optimizers] \label{thm:generalized_Q_symmetric_optimizers} Let \(X \subset \mathbb{H}\) be a $Q$-invariant smooth domain that is convex along $J$-complex lines. Then $X$ achieves equality in the Viterbo inequality if and only if $X$ is symplectomorphic to the standard 4-ball $B \subset \R^4$ with $\on{vol}(B) = \on{vol}(X)$.
\end{thm}

\noindent We require the following result of Abbondandolo-Bramham-Hryniewicz-Salom\~{a}o \cite[Prop 4.3]{abbondandolo2018sharp}. 

\begin{prop}[Zoll Ball] \label{prop:zoll_iff_ball} Let $X \subset \R^4$ be a smooth star-shaped domain. Then $X$ is symplectomorphic to the 4-ball $B \subset \R^4$ with $\on{vol}(B) = \on{vol}(X)$ if and only if the Reeb flow on $\partial X$ is Zoll.\end{prop}

\begin{proof}[Proof of Theorem \ref{thm:generalized_Q_symmetric_optimizers}] It suffices to show that if $X$ achieves equality in the Viterbo inequality, then it is symplectomorphic to a ball since the other direction is trivial. 

\vspace{3pt}

Thus suppose that $X$ achieves equality and let $F$ denote the corresponding Finsler metric on $S^2$, which descends to a Finsler metric $\bar{F}$ on $\mathbb{RP}^2$. Then it follows from (\ref{eq:Q_invariant_vol_bound_sandwich}) that $\bar{F}$ satisfies
\[
2 \cdot \on{sys}(\mathbb{RP}^2,\bar{F})^2 = \on{vol}(\mathbb{RP}^2,\bar{F}) \qquad\text{and}\qquad \sys(X) = 4 \cdot \on{sys}(\mathbb{RP}^2,\bar{F})
\]
Then Theorem \ref{thm:ivanov_pu} states that $\bar{F}$ is Zoll, so that every point is contained in a closed geodesic of minimal period. Moreover, every closed Reeb orbit on the unit cosphere bundle of $\bar{F}$ must lift under the map $\bar{\pi}$ in (\ref{eq:Q_invariant_sys_projection}) to a single orbit of period $\sys(X)$ since otherwise $\sys(X)$ would be strictly less than $4\on{sys}(\mathbb{RP}^2,\bar{F})$. It follows that every point in $\partial X$ is contained in a closed Reeb orbit of minimal period, or equivalently that $\partial X$ is Zoll. Proposition \ref{prop:zoll_iff_ball} yields the result. 
\end{proof}

\begin{remark}[Non-Smooth Domains] Due to the Hausdorff continuity of the symplectic systole among convex sets, we can apply an approximation argument to acquire the systolic inequality in Theorem \ref{thm:generalized_Q_symetric_Viterbo} for convex sets satisfying the appropriate symmetries that are not necessarily smooth. The same remark applies to Theorem \ref{thm:I_invariant_systolic} and Theorem \ref{thm:conditional_I_Viterbo} below.
\end{remark}

\subsection{$I$-Symmetric And Centrally Symmetric Domains.} We next discuss applications of the quaternionic covering to systolic bounds for domains invariant under smaller symmetry groups.
\begin{thm}[$I$-Symmetric Systolic Bound] \label{thm:I_invariant_systolic} Let \(X \subset \mathbb{H}\) be an $I$-invariant smooth domain that is convex along $J$-complex lines. Then $X$ satisfies the following systolic inequality.
\[
\frac{\sys(X)}{\sys(B)} < \big(\frac{3}{2}\big)^{1/2} \cdot \big(\frac{\on{vol}(X)}{\on{vol}(B)}\big)^{1/2}
\]
\end{thm}

\begin{proof} Recall that $I$-invariant domains descend to possibly non-reversible Finsler metrics on $\RP^2$. By applying Proposition \ref{prop:non_reversible_Pu_sota} and an identical argument to the proof of Theorem \ref{thm:generalized_Q_symetric_Viterbo}, we acquire the following chain of inequalities.
\begin{equation} \label{eq:I_invariant_vol_bound_sandwich}
\sys(X)^2 \le 16 \cdot \sys(\mathbb{RP}^2,\bar{F})^2 < 12 \cdot \volHT(\mathbb{RP}^2,\bar{F}) = 3 \cdot \on{vol}(X)\end{equation}
where the middle inequality follows from the non-sharp systolic inequality in Proposition \ref{prop:non_reversible_Pu_sota} for non-reversible Finsler metrics on $\mathbb{RP}^2$. Note that it must be a strict inequality, since the equality case is obtained on metrics whose fibers are not smooth nor strongly convex (triangles). Since $\sys(B)^2 = 2\on{vol}(B)$ for any metric ball $B$, the above inequality is equivalent to the strict inequality. \end{proof}

We note that this result gives a state-of-the-art constant for Viterbo's inequality in the $I$-symmetric setting. The previous best estimate (cf. \cite{gutt2022examples}) is given by the following simple argument, included here for the convenience of the reader.

\begin{thm}[\cite{gutt2022examples}] \label{thm:I_invarian_worse} Let $I$ be the standard complex structure on $\C^n$ and let $X$ be an $I$-invariant convex domain in $\mathbb{C}^n$. Then
\[
c_{EHZ}(X)^n \le \frac{4^n n!}{\pi^n} \cdot \on{vol}(X)
\]
Moreover, there exists a metric ball centered at the origin $B \subset X$ with Ekeland-Hofer-Zehnder capacity
\[
c_{EHZ}(B) \ge (\frac{\pi}{4})^n \cdot c_{EHZ}(X)
\]\end{thm}

\begin{proof} Let $z$ be the boundary point of $X$ that is closest to the origin. Then the ball $B$ of radius $|z|$ is contained in $X$ and intersects $\partial X$ at $z$. In particular, $X$ is contained in the half-space
\[
H = \{w \in \C^n \; : \; \langle w,z\rangle \le |z|^2\} 
\]
After applying a unitary transformation, we may assume that $z = (r,0\dots 0)  \in \C^n$ where $r = |z|$ and thus $H$ is the half-space of vectors $w = (w_1,\dots,w_n)$ with $\on{Re}(w_1) \le r$. Since $X$ is $i$-invariant and $X \subset H$, we have that
\[
X \subset H \cap iH \cap -H \cap -iH = Z \qquad\text{where} \qquad Z = S \times \C^{n-1}
\]
Here $S \subset \C$ is the square consisting of points with real and imaginary parts in the interval $[-r,r]$. Now we simply note that the Ekeland-Hofer-Zehnder capacity of $B$ and $Z$ are given by
\[
c_{EHZ}(B) = \pi r^2 \qquad\text{and}\qquad c_{EHZ}(Z) = 4r^2 
\]
Since the Ekeland-Hofer-Zehnder capacity is monotonic with respect to embeddings, we have
\[
n! \cdot \on{vol}(X) \ge n! \cdot \on{vol}(B) \ge c_{EHZ}(B)^n = (\frac{\pi}{4})^n \cdot c_{EHZ}(Z)^n \ge (\frac{\pi}{4})^n \cdot c_{EHZ}(X)^n
\]
Rearranging this inequality gives the desired volume bound. Moreover, the above sequence of inequalities contains the bound via the Ekeland-Hofer-Zehnder capacity of the ball $B$. \end{proof}

In the convex (and not necessarily smooth) case, we can improve the constant in Theorem \ref{thm:I_invariant_systolic} slightly using a compactness argument. This recovers Theorem \ref{thm:intro_I_invariant_systolic} from the introduction.

\begin{thm}[$I$-Symmetric Systolic Bound, Convex Case] \label{thm:I_invariant_systolic_convex} Let \(X \subset \mathbb{H}\) be an $I$-invariant convex domain. Then $X$ satisfies the following systolic inequality.
\[
\frac{\sys(X)}{\sys(B)} \le C \cdot \big(\frac{\on{vol}(X)}{\on{vol}(B)}\big)^{1/2} \qquad\text{where $C < \big(\frac{3}{2}\big)^{1/2}$.}
\]
\end{thm}

\begin{proof} Consider the following normalized systolic ratio associated to a convex body $X \subset \HH$.
\[
\rho(X) = \frac{\sys(X)}{\sys(B)} \cdot \big(\frac{\on{vol}(B)}{\on{vol}(X)}\big)^{1/2}
\]
This function is continuous on convex bodies (cf. \cite{abbondandolo2018sharp}) and strictly bounded above by $(3/2)^{1/2}$ on $I$-symmetric smooth convex bodies by Theorem \ref{thm:I_invariant_systolic}. Therefore $\rho(X)$ achieves a supremum over all $I$-symmetric convex bodies. Let $X_i$ be a sequence of convex bodies of volume one with
\[\rho(X_i) \to C.\]
By Theorem \ref{thm:I_invarian_worse}, there is a metric ball of capacity bounded below by $C' = (\pi/4)^2 \cdot C - \epsilon$ contained in $X_i$ for some sufficiently small $\epsilon$ and all sufficiently large $i$. By Blaschke's selection theorem, we may pass to a subsequence so that $X_i \to X$ in the Hausdorff topology. Since the volume and the systole (or equivalently, Ekeland-Hofer-Zehnder capacity) are both Hausdorff continuous on convex domains, this implies that $\rho(X) = C$. 

\vspace{3pt}

On the other hand, let $F$ be the Lipschitz Finsler metric associated to $X$ as in Remark \ref{rmk:lipschitz_metrics}. The argument in Theorem \ref{thm:I_invariant_systolic} implies that $C = \rho(X) \le (3/2)^{1/2}$ with equality if and only if the unit codisk fibers of $D^*_F\mathbb{RP}^2$ are triangles. However, $F^*$ is a squared metric on $T^*\mathbb{RP}^2$ by Remark \ref{rmk:lipschitz_metrics}, and thus strongly convex by Lemma \ref{lem:square_stays_convex}, so the unit codisks cannot be triangles. Therefore $C$ is less than $(3/2)^{1/2}$. \end{proof}

\noindent Similarly, by substituting Conjecture \ref{conj:nonreversible_pu} for the Ivanov-Pu inequality in the arguments in Theorems \ref{thm:generalized_Q_symetric_Viterbo} and \ref{thm:generalized_Q_symmetric_optimizers}, we acquire the following conditional result, which implies Theorem~\ref{thm:intro_conditional} from the introduction.

\begin{thm}[Conditional $I$-Symmetric Viterbo] \label{thm:conditional_I_Viterbo} Assume Conjecture \ref{conj:nonreversible_pu}. Then any $I$-invariant smooth domain $X \subset \HH$ that is convex along $J$-complex lines satisfies the Viterbo inequality, with equality if and only if $X$ is symplectomorphic to the 4-ball. 
\end{thm}

Finally, we prove the general systolic inequality satisfied by centrally symmetric domains convex along $J$-complex lines discussed in the introduction (Theorem \ref{thm:intro_sys_ineq_fiberwise_convex}). 

\begin{thm}[Theorem \ref{thm:intro_sys_ineq_fiberwise_convex}] \label{thm:sys_ineq_fiberwise_convex} Let $X$ be a smooth, centrally symmetric domain in $\HH$ that is convex along $J$-complex lines. Then the length $\sys(X)$ of the shortest Reeb orbit on $\partial X$ satisfies
\[
\frac{\sys(X)}{\sys(B)} \le C \Big( \frac{\on{vol}(X)}{\on{vol}(B)}\Big)^{1/2} \qquad\text{where}\qquad C \le  8 + \frac{11\sqrt{3}}{2}
\]
\end{thm}

\begin{proof} Let $F$ be the corresponding Finsler metric to $X$ on $S^2$ via Corollary \ref{cor:quaternion_cover}. By Theorem \ref{thm:shortest_geodesic_ineq}, we have the following inequality for the shortest geodesic length.
\[
\on{sg}(S^2,F)^2  \le C^2 \cdot \volHT(S^2,F) \qquad\text{where}\qquad C \le 8 + \frac{11\sqrt{3}}{2}
\]
The shortest geodesic of $F$ lifts to a closed Reeb orbit on $\partial X$ of length at most $2 \cdot \on{sg}(S^2,F)$ and the volume of $X$ is twice the Holmes-Thompson volume of $F$ by our convention. Therefore
\[
\sys(X)^2 \le 4 \cdot \on{sg}(S^2,F)^2 \le 4C^2 \cdot \volHT(S^2,F) = 2C^2 \cdot \on{vol}(X)
\]
Since $\sys(B)^2/\on{vol}(B) = 2$ for the 4-dimensional symplectic ball $B$, this yields the result.\end{proof}

\subsection{Minimizing Reeb Orbits And Hopfness} We conclude this section by proving our main results on Hopfness of links of orbits.  We start by proving the following correspondence between minimal period Reeb orbits and systolic (i.e. non-contractible and length minimizing) geodesics.

\begin{thm}[Minimizer Bijection] \label{thm:minimizer_bijection} Let $X$ be a $Q$-symmetric smooth, uniformly convex domain in $\HH$. Then $\Gamma \subset \partial X$ is a minimal period Reeb orbit if and only if it is the lift of a systolic geodesic
\[
\gamma \subset \mathbb{RP}^2 \quad\text{of the associated Finsler metric $F$ under the projection}\quad \rho:\partial X \to S^*_F\mathbb{RP}^2 \to \mathbb{RP}^2.
\]
\end{thm}

\noindent The proof of this result will require the following theorem of Akopyan-Karasev \cite[Thm 2.3]{akopyan2017estimating}.

\begin{thm} \cite{akopyan2017estimating} \label{thm:karasev}  Let $X \subset \mathbb{C}^n$ be a smooth and uniformly convex body such that $uX = X$ for a primitive $m$th root of unity $u$. Then any closed characteristic of minimal period on the boundary $\partial X$ is invariant under multiplication by $u$.
\end{thm}

\begin{proof}[Proof of Theorem \ref{thm:minimizer_bijection}] Let $\Gamma$ be a closed Reeb orbit in $\partial X$ and consider the projection
\[
\pi(\Gamma) \quad\text{under the projection}\quad \pi:\partial X \to S^*_F\mathbb{RP}^2
\]
Note that $\pi(\Gamma)$ is the lift of a geodesic $c_\Gamma \subset \mathbb{RP}^2$. Moreover, $\Gamma$ is invariant under multiplication by $I$ if and only if the homology class $[\pi(\Gamma)]$ is a generator of $H_1(S^*_F\mathbb{RP}^2;\Z) = \Z/4$, which occurs if and only if $c_\Gamma$ is non-contractible. This determines a bijection $\Gamma \mapsto c_\Gamma$ between $I$-invariant closed Reeb orbits on $\partial X$ and non-contractible geodesics of $F$, such that the period of $\Gamma$ is four times the length of $c_\Gamma$.

\vspace{3pt}

Now we argue the result as follows. Let $\Gamma$ be a closed Reeb orbit. Note that $X$ is invariant under a primitive 4th root of unity, namely $I = i$ under the identification $\HH = \C^2$. Thus by Theorem \ref{thm:karasev}, $\Gamma$ minimizes the period if and only if $\Gamma$ is $I$-invariant and $\Gamma$ minimizes the period among $I$-invariant closed Reeb orbits. By the correspondence above, this occurs if and only if $\Gamma$ is the lift of a closed non-contractible geodesic $\gamma = c_\Gamma$ of minimal period. \end{proof}

\begin{remark} We do not know how to prove this result in the more general case of a $Q$-symmetric domain that is convex along $J$-complex lines. The essential issue is that it is (a priori) possible for a minimal Reeb orbit to fail to be symmetric in that case.
\end{remark}

\noindent We are now ready to restate and prove Theorem \ref{thm:intro_hopf_link} from the introduction.

\begin{thm} \label{thm:body_hopf_link} Let $X$ be a $Q$-symmetric smooth, uniformly convex domain in $\HH$ and let $L \subset \partial X$ be any link consisting of minimal period closed Reeb orbits. Then $L$ is Hopf.
\end{thm}

\begin{proof} Let $F$ be the corresponding Finsler metric to $X$ on $\mathbb{RP}^2$ and consider the contact covering $\pi:\partial X \to S^*_F\mathbb{RP}^2$ given by Corollary \ref{cor:quaternion_cover}. Then
\[
\pi(L) \subset S^*_F\mathbb{RP}^2
\]
is a link whose components are cotangent lifts of systolic geodesics by Theorem \ref{thm:minimizer_bijection}. Since $\pi$ is a contact covering, any transverse isotopy of $\pi(L)$ lifts to a transverse isotopy of $L$. It suffices to show that $\pi(L)$ is transversely isotopic to the cotangent lifts of a union of great circles in $\mathbb{RP}^2$ (i.e. geodesics of the round metric). By Lemma~\ref{lem:cot_lift_idep_F}, the lifts are independent of the metric, and therefore transversely isotopic to lifts of great circles under the round metric, which in turn lift to Hopf fibers under $\pi$. 

\vspace{3pt}

To prove this, let $\gamma_1 \dots \gamma_n$ be the distinct unoriented geodesics of $F$ in $\mathbb{RP}^2$ given by projecting the components of $L$. Note that two components of $L$ might project to the same geodesic if they agree up to reversal of orientation. Let
\[
\Xi \subset S^*_F\mathbb{RP}^2
\]
be the $2n$-component transverse link consisting of the union of the transverse lifts of $\gamma_i$ with both orientations for each $i = 1 \dots n$. Note that $\pi(L)$ is a sub-link of $\Xi$, or equivalently a union of components of $\Xi$. By Lemma \ref{lem:embedded_systoles} and Lemma \ref{lem:systolic_intersections}, the geodesics $\gamma_i$ are embedded and intersect once pairwise. Therefore, by Proposition \ref{prop:config_immersely_isotopic}, this configuration is isotopic without tangencies to a union of great circles in $\mathbb{RP}^2$. This isotopy then lifts to a transverse isotopy of $\Xi$ to a union of cotangent lifts of great circles. This isotopy restricts to a transverse isotopy of $\pi(L)$, which yields the desired isotopy in the first part. 
\end{proof}

We conclude by presenting a new proof of the following result of Abbondandolo-Edtmair-Kang \cite{abbondandolo2024closed} in the centrally symmetric case.

\begin{thm}[Hopf Minimizer] \label{thm:body_Hopf_minimizer} Let $X$ be a centrally symmetric, smooth, uniformly convex domain in $\HH$ and let $\Gamma \subset \partial X$ be a minimal period closed Reeb orbit. Then $\Gamma$ is Hopf.
\end{thm}

\noindent The proof uses the Clarke dual formulation of the Ekeland-Hofer-Zehnder capacity (cf. \cite{artstein2008brunn,artstein2014bounds,haim2019symplectic}). Let us briefly recall the basic notation and results from Clarke duality.

\begin{remark}[Clarke Duality] \label{rmk:clarke_duality} Consider $\mathbb{R}^{2n}$ equipped with the standard symplectic form $\Omega$, standard radial Liouville form $\lambda_\HH$ and standard complex structure $I$. Given a smooth convex body $X \subset \R^{2n}$ containing the origin, denote the support function (or dual norm) by
\[
h_X:\R^{2n} \to [0,\infty)
\]The Clarke dual action $\mathcal{C}_X$ and the symplectic action $\mathcal{A}(z)$ of a piecewise smooth curve $z:[0,S] \to \R^{2n}$ are respectively given by
\[
\mathcal{C}_X(z) = \frac{1}{4} \cdot \int_0^S h_X(-I\dot{z}(t))^2 dt \qquad\text{and}\qquad \mathcal{A}(z) = \int_z \lambda_\HH
\]
The Ekeland-Hofer-Zehnder capacity of $X$ is then given by the following infimum.
\[
\sys(X) = \on{inf}\big\{\frac{\mathcal{C}_X(z)}{\mathcal{A}(z)} \; : \text{piecewise smooth } \; z:\R/\Z \to \R^{2n} \text{ with }\mathcal{A}(z) > 0\big\}
\]

The Clarke action of a Reeb segment with the appropriate parametrization takes a particularly simple form. Precisely, recall that the Reeb vector-field on $\partial X$ is given by
\[
R(x) = \frac{2I\nu(x)}{\langle x,\nu(x)\rangle}
\]
where $\nu(x)$ is the outward unit normal vector-field of $\partial X$ and $\langle-,-\rangle$ is the standard inner product. Moreover, for a point $x$ in $\partial X$, we have
\[
h_X(\nu(x)) = \langle \nu(x),x\rangle
\]
Therefore, if $z:[0,S] \to \partial X$ is a segment of trajectory of the Reeb vector-field, parametrized so that $\dot{z} = TR \circ z$ for $T > 0$, then the Clarke action is given by
\[
\mathcal{C}_X(z) = \frac{1}{4}\cdot \int_0^S h_X(-I\dot{z}(t))^2 dt = \frac{1}{4}\cdot \int_0^S h_X(-ITR(z(t)))^2 dt
\]
\[
= \frac{1}{4}\cdot \int_0^S h_X\Big(\frac{2T\nu(z(t))}{h_X(\nu(z(t)))}\Big)^2 dt = \frac{4T^2}{4} \cdot \int_0^S \frac{h_X(\nu(z(t)))^2}{h_X(\nu(z(t)))^2} dt = S \cdot T^2 
\]
Note that the same computation holds for any translate of a Reeb segment with the same parametrization.
\end{remark}

The key input from Clarke duality that we will need is the following bound for the Ekeland-Hofer-Zehnder capacity for a centrally symmetric domain. Similar arguments appeared in \cite{akopyan2017estimating} and \cite{jin2023coisotropic}.

\begin{lemma}[Chord Bound] \label{lem:chord_bound} Let $X$ be a centrally symmetric, smooth, uniformly convex domain in $\R^{2n}$ and let $c$ be a Reeb chord on $\partial X$ of action $T$ with endpoints on a Lagrangian sub-space $L \subset \R^{2n}$. Then
\[
2T \ge \sys(X)
\]
\end{lemma}

\begin{proof} Suppose that $c$ is a Reeb chord on $\partial X$ beginning and ending on $L$, parametrized by $[0,1/2]$ so that $\dot{c} = 2TR \circ c$ where $T = \mathcal{A}(c)$. Let $z$ be the loop parametrized by $\R/\Z$ such that
\[
z(t) = c(t) - \frac{c(1/2) + c(0)}{2} \text{ for $t \in [0,1/2)$}\qquad \text{and}\qquad z(t) = -z(t-1/2) \text{ for $t \in [1/2,1)$}
\]
Note that $z(t)$ approaches $-z(0)$ as $t \to 1/2$ so that the loop is continuous. Moreover, the Clarke action is invariant under translation, and additionally invariant under the involution $x \mapsto -x$ due to the central symmetry of $X$. Thus the computation of the action in Remark \ref{rmk:clarke_duality} yields
\[
\mathcal{C}_X(z) = 2 \cdot \mathcal{C}_X(c) = 2 \cdot \frac{1}{2} \cdot (2\mathcal{A}(c))^2 = 4\mathcal{A}(c)^2
\]
Finally, we compute the symplectic action of $z$. Let $w = (c(1/2) + c(0))/2$. Using the usual formula for the standard Liouville form and the central symmetry of $X$, we have that
\[
\mathcal{A}(z) = \frac{1}{2} \int_0^1 \Omega(z(t),\dot{z}(t)) dt = \int_0^{1/2} \Omega(z(t),\dot{z}(t)) dt = \int_0^{1/2} \Omega(c(t),\dot{c}(t)) dt - \int_0^{1/2} \Omega(w,\dot{c}(t)) dt
\]
\[
= 2\mathcal{A}(c) - \Omega(w,c(1/2) - c(0))
\]
Now simply note that $w$ and $c(1/2) - c(0)$ are both contained in the Lagrangian subspace $L$ containing $c(1/2)$ and $c(0)$. Therefore the second term above vanishes and $\mathcal{A}(z) = 2\mathcal{A}(c)$. Thus
\[
\sys(X) \le \frac{\mathcal{C}_X(z)}{\mathcal{A}(z)} = \frac{4\mathcal{A}(c)^2}{2\mathcal{A}(c)} = 2\mathcal{A}(c) \qedhere
\]\end{proof}

\noindent We can now prove the main result on minimal orbits of centrally symmetric domains.

\begin{proof}[Proof of Theorem \ref{thm:body_Hopf_minimizer}] Let $\Gamma$ be a minimizing Reeb trajectory on the boundary of a smooth uniformly convex and centrally symmetric domain $X \subset \HH$.  We first claim that $\Gamma$ can intersect any $J$-complex line $L \subset \HH$ at most twice. Indeed, suppose otherwise. Then $\Gamma$ decomposes into a composition
\[\Gamma = c_1 \star c_2 \star \dots \star c_k\]
where $k > 2$ and $c_i$ is a Reeb chord beginning and ending on $L$. Since $L$ is a Lagrangian subspace, we may apply Lemma \ref{lem:chord_bound} to find that
\[
\on{sys}(X) = \mathcal{A}(\Gamma) = \sum_i \mathcal{A}(c_i) \ge \frac{k}{2} \cdot \on{sys}(X)
\]
This is a contradiction. This shows that $\Gamma$ must intersect $L$ at most twice.

\vspace{3pt}

Now let $\gamma$ be the corresponding geodesic in $S^2$ of the associated Finsler metric $F$ to $X$ and consider the cotangent lift $\widetilde{\gamma}$ of $\gamma$ to $S^*_FS^2 \subset T^*S^2$. Recall that the cotangent lift is double covered by $\Gamma$ via the covering map from Theorem \ref{thm:quaternion_cover},
\[
\pi:\HH^\times \to T^*S^2.
\]
Here we used the fact that $\Gamma$ is centrally symmetric due to Theorem~\ref{thm:karasev}.
Moreover, the inverse image of a cotangent fiber under $\pi$ is precisely a $J$-complex line minus the origin. Therefore, $\widetilde{\gamma}$ intersects each cotangent fiber at most once by the previous claim. Therefore the projection $\gamma$ to $S^2$ is embedded.  Corollary \ref{cor:tangent_lifts_trans_isot} then implies that $\gamma$ is isotopic to a great circle, and thus the lift $\Gamma$ to $\partial X$ is isotopic to the lift of a great circle. Since the lift is independent of the metric up to transverse isotopy (Lemma~\ref{lem:cot_lift_idep_F}), it is a Hopf fiber.\end{proof}

\bibliographystyle{hplain}
\bibliography{dyn_cvx_2_bib}

\end{document}